\documentclass[english]{article}
\usepackage[letterpaper,top=2.5cm,bottom=3cm,left=2.5cm,right=2.5cm,marginparwidth=1.8cm]{geometry}
\usepackage{babel}
\usepackage{verbatim}
\usepackage{mathtools}
\usepackage{booktabs}
\usepackage{enumitem}
\usepackage{bm}
\usepackage{algorithm, algpseudocode, multirow}
\providecommand{\algorithmname}{Algorithm}
\floatname{algorithm}{\protect\algorithmname}
\floatstyle{ruled}
\newfloat{function}{tbp}{lofn}
\floatname{function}{Function}
\usepackage{amsmath}
\numberwithin{equation}{section}
\usepackage{amsthm}
\usepackage{amssymb}
\usepackage{placeins}
\usepackage{microtype}
\usepackage{color,xcolor}
\usepackage{wrapfig}
\usepackage{pifont}
\usepackage[numbers]{natbib}
\usepackage[pdfusetitle, colorlinks=false]{hyperref}
\hypersetup{pdfauthor={Zhenwei Lin and Zhe Zhang}}
\usepackage[all]{hypcap}
\usepackage{cleveref}
\crefname{ALG@line}{line}{lines}
\Crefname{ALG@line}{Line}{Lines}
\makeatletter
\@ifpackageloaded{hyperref}{
    \DeclareRobustCommand{\theHALG@line}{line.\thealgorithm.\thefunction.\arabic{ALG@line}}
}{
}
\makeatother

\makeatletter
\theoremstyle{definition}
\newtheorem{defn}{\protect\definitionname}[]
\theoremstyle{plain}
\newtheorem{thm}{\protect\theoremname}[section]
\theoremstyle{plain}
\newtheorem{proposition}{\protect\propositionname}[section]
\theoremstyle{plain}
\newtheorem{assumption}{\protect\assumptionname}
\theoremstyle{plain}
\newtheorem{lem}{\protect\lemmaname}[section]
\theoremstyle{remark}
\newtheorem{rem}{\protect\remarkname}[section]
\theoremstyle{definition}

\providecommand{\examplename}{Example}
\theoremstyle{plain}

\providecommand{\assumptionname}{Assumption}
\providecommand{\definitionname}{Definition}
\providecommand{\lemmaname}{Lemma}
\providecommand{\propositionname}{Proposition}
\providecommand{\remarkname}{Remark}
\providecommand{\theoremname}{Theorem}
\providecommand{\corollaryname}{Corollary}

\crefdefaultlabelformat{#2\textbf{#1}#3} %
\crefname{section}{\textbf{section}}{\textbf{sections}}
\Crefname{section}{\textbf{Section}}{\textbf{Sections}}
\crefname{thm}{\textbf{Theorem}}{\textbf{theorems}}
\Crefname{thm}{\textbf{Theorem}}{\textbf{Theorems}}
\crefname{lem}{\textbf{Lemma}}{\textbf{lemmas}}
\Crefname{lem}{\textbf{Lemma}}{\textbf{Lemmas}}
\crefname{prop}{\textbf{proposition}}{\textbf{propositions}}
\Crefname{prop}{\textbf{Proposition}}{\textbf{Propositions}}
\crefname{algorithm}{\textbf{algorithm}}{\textbf{algorithms}}
\Crefname{algorithm}{\textbf{Algorithm}}{\textbf{Algorithms}}
\crefname{function}{\textbf{function}}{\textbf{functions}}
\Crefname{function}{\textbf{Function}}{\textbf{Functions}}
\crefname{coro}{\textbf{Corollary}}{\textbf{corollaries}}
\Crefname{coro}{\textbf{Corollary}}{\textbf{corollaries}}
\crefname{defn}{\textbf{Definition}}{\textbf{definitions}}
\Crefname{defn}{\textbf{Definition}}{\textbf{definitions}}
\crefname{table}{\textbf{Table}}{\textbf{tables}}
\Crefname{table}{\textbf{Table}}{\textbf{tables}}
\crefname{figure}{\textbf{Figure}}{\textbf{figures}}
\Crefname{figure}{\textbf{Figure}}{\textbf{figures}}
\crefname{exple}{\textbf{Example}}{\textbf{examples}}
\Crefname{exple}{\textbf{Example}}{\textbf{examples}}
\Crefname{assumption}{\textbf{Assumption}}{\textbf{Assumptions}}
\crefname{assumption}{\textbf{Assumption}}{\textbf{Assumptions}}
\Crefname{rem}{\textbf{Remark}}{\textbf{Remarks}}
\crefname{rem}{\textbf{Remark}}{\textbf{Remarks}}

\usepackage[most]{tcolorbox}
\usepackage{fancybox,subfig}
\usepackage[normalem]{ulem}
\tcbuselibrary{breakable,theorems,skins}

\newtcbtheorem[number within=section]{exbox}{Example}{
  enhanced,
  colback=gray!3,           
  colframe=gray!70!black,   
  boxrule=0.6pt,
  arc=2mm,                  
  left=8pt,right=8pt,top=8pt,bottom=8pt,
  fonttitle=\bfseries,       
  attach boxed title to top left={yshift=-1mm, xshift=2mm},
}{ex}

\providecommand{\corollaryname}{Corollary}

\makeatother

\begin{document}
\global\long\def\inprod#1#2{\left\langle #1,#2\right\rangle }%

\global\long\def\inner#1#2{\left\langle #1,#2\right\rangle }%

\global\long\def\binner#1#2{\big\langle#1,#2\big\rangle}%

\global\long\def\norm#1{\left\Vert #1\right\Vert }%

\global\long\def\bnorm#1{\big\Vert#1\big\Vert}%

\global\long\def\Bnorm#1{\Big\Vert#1\Big\Vert}%

\global\long\def\red#1{\textcolor{red}{#1}}%

\global\long\def\blue#1{\textcolor{blue}{#1}}%

\global\long\def\brbra#1{\left(#1\right)}%

\global\long\def\Brbra#1{\left(#1\right)}%

\global\long\def\rbra#1{(#1)}%

\global\long\def\lrbra#1{\left(#1\right)}%

\global\long\def\sbra#1{[#1]}%

\global\long\def\bsbra#1{\left[#1\right]}%

\global\long\def\Bsbra#1{\Big[#1\Big]}%

\global\long\def\abs#1{\vert#1\vert}%

\global\long\def\babs#1{\big\vert#1\big\vert}%

\global\long\def\lrabs#1{\left|#1\right|}%

\global\long\def\cbra#1{\{#1\}}%

\global\long\def\bcbra#1{\left\{  #1\right\}  }%

\global\long\def\Bcbra#1{\left\{  #1\right\}  }%

\global\long\def\matr#1{\bm{#1}}%

\global\long\def\til#1{\tilde{#1}}%

\global\long\def\wtil#1{\widetilde{#1}}%

\global\long\def\wh#1{\widehat{#1}}%

\global\long\def\mcal#1{\mathcal{#1}}%

\global\long\def\mbb#1{\mathbb{#1}}%

\global\long\def\mtt#1{\mathtt{#1}}%

\global\long\def\ttt#1{\texttt{#1}}%

\global\long\def\dtxt{\textrm{d}}%

\global\long\def\aeq{\overset{(a)}{=}}%

\global\long\def\bignorm#1{\bigl\Vert#1\bigr\Vert}%

\global\long\def\Bignorm#1{\Bigl\Vert#1\Bigr\Vert}%

\global\long\def\rmn#1#2{\mathbb{R}^{#1\times#2}}%

\global\long\def\deri#1#2{\frac{d#1}{d#2}}%

\global\long\def\pderi#1#2{\frac{\partial#1}{\partial#2}}%

\global\long\def\limk{\lim_{k\rightarrow\infty}}%

\global\long\def\trans{\textrm{T}}%

\global\long\def\onebf{\mathbf{1}}%

\global\long\def\zerobf{\mathbf{0}}%

\global\long\def\zero{\bm{0}}%

% expectation

\global\long\def\Euc{\mathrm{E}}%

\global\long\def\Expe{\mathbb{E}}%

\global\long\def\rank{\mathrm{rank}}%

\global\long\def\range{\mathrm{range}}%

\global\long\def\diam{\mathrm{diam}}%

\global\long\def\epi{\mathrm{epi} }%

\global\long\def\inte{\operatornamewithlimits{int}}%

\global\long\def\dist{\operatornamewithlimits{dist}}%

\global\long\def\proj{\operatornamewithlimits{Proj}}%

\global\long\def\cov{\mathrm{Cov}}%

\global\long\def\argmin{\operatornamewithlimits{argmin}}%

\global\long\def\argmax{\operatornamewithlimits{argmax}}%

\global\long\def\where{\operatornamewithlimits{where}}%

\global\long\def\conv{\operatornamewithlimits{conv}}%

\global\long\def\tr{\operatornamewithlimits{tr}}%

\global\long\def\dist{\operatorname{dist}}%

\global\long\def\sign{\operatornamewithlimits{sign}}%

\global\long\def\prob{\mathrm{Prob}}%

\global\long\def\st{\operatornamewithlimits{s.t.}}%

\global\long\def\dom{\mathrm{dom}}%

\global\long\def\prox{\mathrm{prox}}%

\global\long\def\diag{\mathrm{diag}}%

\global\long\def\and{\mathrm{and}}%

\global\long\def\as{\textup{a.s.}}%

\global\long\def\ae{\textup{a.e.}}%

\global\long\def\Var{\operatornamewithlimits{Var}}%

\global\long\def\Cov{\operatornamewithlimits{Cov}}%

\global\long\def\raw{\rightarrow}%

\global\long\def\law{\leftarrow}%

\global\long\def\Raw{\Rightarrow}%

\global\long\def\Law{\Leftarrow}%

\global\long\def\vep{\varepsilon}%

\global\long\def\dom{\operatornamewithlimits{dom}}%

\global\long\def\err{\operatorname{err}}%

\global\long\def\soc{\operatorname{soc}}%

\global\long\def\rsoc{\operatorname{rsoc}}%

\global\long\def\tsum{{\textstyle {\sum}}}%

\global\long\def\Cbb{\mathbb{C}}%

\global\long\def\Ebb{\mathbb{E}}%

\global\long\def\Fbb{\mathbb{F}}%

\global\long\def\Nbb{\mathbb{N}}%

\global\long\def\Rbb{\mathbb{R}}%

\global\long\def\extR{\widebar{\mathbb{R}}}%

\global\long\def\Pbb{\mathbb{P}}%

\global\long\def\Mrm{\mathrm{M}}%

\global\long\def\Acal{\mathcal{A}}%

\global\long\def\Bcal{\mathcal{B}}%

\global\long\def\Ccal{\mathcal{C}}%

\global\long\def\Dcal{\mathcal{D}}%

\global\long\def\Ecal{\mathcal{E}}%

\global\long\def\Fcal{\mathcal{F}}%

\global\long\def\Gcal{\mathcal{G}}%

\global\long\def\Hcal{\mathcal{H}}%

\global\long\def\Ical{\mathcal{I}}%

\global\long\def\Kcal{\mathcal{K}}%

\global\long\def\Lcal{\mathcal{L}}%

\global\long\def\Mcal{\mathcal{M}}%

\global\long\def\Ncal{\mathcal{N}}%

\global\long\def\Ocal{\mathcal{O}}%

\global\long\def\Pcal{\mathcal{P}}%

\global\long\def\Scal{\mathcal{S}}%

\global\long\def\Tcal{\mathcal{T}}%

\global\long\def\Xcal{\mathcal{X}}%

\global\long\def\Ycal{\mathcal{Y}}%

\global\long\def\Zcal{\mathcal{Z}}%

\global\long\def\i{i}%
% Use bold text symbol for vector and matrix

\global\long\def\abf{\mathbf{a}}%

\global\long\def\bbf{\mathbf{b}}%

\global\long\def\cbf{\mathbf{c}}%

\global\long\def\ebf{\mathbf{e}}%

\global\long\def\fbf{\mathbf{f}}%

\global\long\def\hbf{\mathbf{h}}%

\global\long\def\qbf{\mathbf{q}}%

\global\long\def\gbf{\mathbf{g}}%

\global\long\def\lambf{\bm{\lambda}}%

\global\long\def\alphabf{\bm{\alpha}}%

\global\long\def\sigmabf{\bm{\sigma}}%

\global\long\def\thetabf{\bm{\theta}}%

\global\long\def\deltabf{\bm{\delta}}%

\global\long\def\lbf{\mathbf{l}}%

\global\long\def\ubf{\mathbf{u}}%

\global\long\def\pbf{\mathbf{\mathbf{p}}}%

\global\long\def\vbf{\mathbf{v}}%

\global\long\def\wbf{\mathbf{w}}%

\global\long\def\xbf{\mathbf{x}}%

\global\long\def\ybf{\mathbf{y}}%

\global\long\def\zbf{\mathbf{z}}%

\global\long\def\dbf{\mathbf{d}}%

\global\long\def\Wbf{\mathbf{W}}%

\global\long\def\Abf{\mathbf{A}}%

\global\long\def\Gbf{\mathbf{G}}%

\global\long\def\Ubf{\mathbf{U}}%

\global\long\def\Pbf{\mathbf{P}}%

\global\long\def\Ibf{\mathbf{I}}%

\global\long\def\Ebf{\mathbf{E}}%

\global\long\def\Mbf{\mathbf{M}}%

\global\long\def\Dbf{\mathbf{D}}%

\global\long\def\Qbf{\mathbf{Q}}%

\global\long\def\Lbf{\mathbf{L}}%

\global\long\def\Pbf{\mathbf{P}}%

\global\long\def\Xbf{\mathbf{X}}%
% Use bold symbol for vector and matrix

\global\long\def\abm{\bm{a}}%

\global\long\def\bbm{\bm{b}}%

\global\long\def\cbm{\bm{c}}%

\global\long\def\dbm{\bm{d}}%

\global\long\def\ebm{\bm{e}}%

\global\long\def\fbm{\bm{f}}%

\global\long\def\gbm{\bm{g}}%

\global\long\def\hbm{\bm{h}}%

\global\long\def\pbm{\bm{p}}%

\global\long\def\qbm{\bm{q}}%

\global\long\def\rbm{\bm{r}}%

\global\long\def\sbm{\bm{s}}%

\global\long\def\tbm{\bm{t}}%

\global\long\def\ubm{\bm{u}}%

\global\long\def\vbm{\bm{v}}%

\global\long\def\wbm{\bm{w}}%

\global\long\def\xbm{\bm{x}}%

\global\long\def\ybm{\bm{y}}%

\global\long\def\zbm{\bm{z}}%

\global\long\def\Abm{\bm{A}}%

\global\long\def\Bbm{\bm{B}}%

\global\long\def\Cbm{\bm{C}}%

\global\long\def\Dbm{\bm{D}}%

\global\long\def\Ebm{\bm{E}}%

\global\long\def\Fbm{\bm{F}}%

\global\long\def\Gbm{\bm{G}}%

\global\long\def\Hbm{\bm{H}}%

\global\long\def\Ibm{\bm{I}}%

\global\long\def\Jbm{\bm{J}}%

\global\long\def\Lbm{\bm{L}}%

\global\long\def\Obm{\bm{O}}%

\global\long\def\Pbm{\bm{P}}%

\global\long\def\Qbm{\bm{Q}}%

\global\long\def\Rbm{\bm{R}}%

\global\long\def\Ubm{\bm{U}}%

\global\long\def\Vbm{\bm{V}}%

\global\long\def\Wbm{\bm{W}}%

\global\long\def\Xbm{\bm{X}}%

\global\long\def\Ybm{\bm{Y}}%

\global\long\def\Zbm{\bm{Z}}%

\global\long\def\lambm{\bm{\lambda}}%

\global\long\def\alphabm{\bm{\alpha}}%

\global\long\def\albm{\bm{\alpha}}%

\global\long\def\taubm{\bm{\tau}}%

\global\long\def\mubm{\bm{\mu}}%

\global\long\def\yrm{\mathrm{y}}%

\global\long\def\rone{\text{\ensuremath{\brbra{\textrm{I}}}}}%

\global\long\def\rtwo{\brbra{\text{II}}}%

\global\long\def\rthree{\brbra{\text{\textrm{III}}}}%

\global\long\def\rfour{\brbra{\text{\textrm{IV}}}}%

\global\long\def\rfive{\brbra{\text{V}}}%

\global\long\def\rsix{\brbra{\text{\textrm{VI}}}}%

\global\long\def\rseven{\brbra{\text{VI\textrm{I}}}}%

\global\long\def\reight{\brbra{\text{VI\textrm{I}I}}}%

\global\long\def\aleq{\overset{(a)}{\leq}}%

\global\long\def\bleq{\overset{(b)}{\leq}}%

\global\long\def\cleq{\overset{(c)}{\leq}}%

\global\long\def\dleq{\overset{(d)}{\leq}}%

\global\long\def\ageq{\overset{(a)}{\geq}}%

\global\long\def\bgeq{\overset{(b)}{\geq}}%

\global\long\def\cgeq{\overset{(c)}{\geq}}%

\global\long\def\beq{\overset{(b)}{=}}%

\global\long\def\ceq{\overset{(c)}{=}}%

\global\long\def\deq{\overset{(d)}{=}}%

\global\long\def\vbfp{\vbf_{\text{p}}}%

\global\long\def\vbfd{\vbf_{\text{d}}}%

\global\long\def\tp{t_{\text{p}}}%

\global\long\def\td{t_{\text{d}}}%

\global\long\def\tab{\qquad}%

\global\long\def\btab{\hspace{1.2cm}}%

\global\long\def\bbtab{\hspace{1.8cm}}%

\global\long\def\Lin{\operatorname{Lin}}%

\global\long\def\Span{\operatorname{Span}}%

\global\long\def\supp{\operatorname{supp}}%

\global\long\def\holder{\text{Hölder}}%
\global\long\def\apex{\text{APEX}}%
\global\long\def\pws{\text{PWS}}%
\global\long\def\rapex{\text{r}\apex}%
\global\long\def\flag{\text{\textbf{Flag}}}%
\global\long\def\false{\text{\textbf{\text{False}}}}%
\global\long\def\true{\text{\textbf{True}}}%
\global\long\def\Wcer{\text{W-certificate}}%

\global\long\def\Adet{\mbb A_{\text{det}}}%
\global\long\def\Azr{\mbb A_{\text{zr}}}%
\global\long\def\calZA{\mcal Z_{\mcal A}}%
\global\long\def\onestep{\text{One-Step}}%
\global\long\def\quarflag{\text{\textbf{Cert-Flag}}}%
\global\long\def\wrapex{\texttt{wrAPEX}}%
\global\long\def\maxquad{\texttt{MAXQUAD}}%

\newcommand{\jim}[1]{\textcolor{red}{\textbf{#1}}}
\newcommand{\lzw}[1]{\textcolor{blue}{\textbf{#1}}}

\global\long\def\holder{\text{Hölder}}%
\global\long\def\apps{\texttt{APPS}}%
\global\long\def\pws{\text{PWS}}%
\global\long\def\rapps{\text{r}\apps}%
\global\long\def\flag{\text{\textbf{Flag}}}%
\global\long\def\false{\text{\textbf{\text{False}}}}%
\global\long\def\true{\text{\textbf{True}}}%
\global\long\def\Wcer{\text{W-certificate}}%
\global\long\def\douflag{\text{\textbf{Dou-Flag}}}%

\global\long\def\Adet{\mbb A_{\text{det}}}%
\global\long\def\Azr{\mbb A_{\text{zr}}}%
\global\long\def\calZA{\mcal Z_{\mcal A}}%
\global\long\def\onestep{\text{One-step}}%
\global\long\def\halfflag{\text{\textbf{Half-Flag}}}%
\global\long\def\paramflag{\text{\textbf{Param-Flag}}}%
\global\long\def\none{\textbf{None}}%
\global\long\def\PWcer{\text{Penalty W-certificate}}%
\global\long\def\PWcerpair{\text{Paired-cut Penalty W-certificate}}%
\global\long\def\onestepplus{\text{One-Step}^{+}}%
\global\long\def\apexpplus{\texttt{APEX}^{+}}%
\global\long\def\xxx{\red{xxxx}}%
\global\long\def\adjustPara{\texttt{ParamAdjust}}%
\global\long\def\holder{\text{Hölder}}%
\global\long\def\apex{\text{APEX}}%
\global\long\def\pws{\text{PWS}}%
\global\long\def\rapex{\text{r}\apex}%
\global\long\def\flag{\text{\textbf{Flag}}}%
\global\long\def\Wcer{\text{W-certificate}}%
\global\long\def\halfflag{\text{\textbf{Half-Flag}}}%
\global\long\def\lowerflag{\text{\textbf{Lower-Flag}}}%
\global\long\def\Adet{\mbb A_{\text{det}}}%
\global\long\def\Azr{\mbb A_{\text{zr}}}%
\global\long\def\calZA{\mcal Z_{\mcal A}}%
\global\long\def\onestep{\text{One-Step}}%
\global\long\def\quadflag{\text{\textbf{Quad-Flag}}}%
\global\long\def\wrapex{\text{wrAPEX}}%
\global\long\def\xxxx{\red{xxxx}}%
\global\long\def\rapexW{\text{rAPEX-W}}%
\global\long\def\rapexC{\text{rAPEX-C}}%
\global\long\def\convexCert{\text{6.2}}%
\global\long\def\papex{\text{Penalty APEX}}%
\global\long\def\constrNum{{m}}
\global\long\def\activeconstrNum{\abs{\mcal A}}
\global\long\def\activeconstrSet{{\mcal A}}%
\global\long\def\guessTime{\mathcal{T}}

\global\long\def\bundleSize{B}%
\global\long\def\Cert{\text{Cert}}%
\global\long\def\pawg{\mathcal{AWG}^+}
\global\long\def\pagr{\mathcal{AGR}^+}
\global\long\def\pawgl{\mathcal{AWGL}^+}
\global\long\def\pagrl{\mathcal{AGRL}^+}
\global\long\def\activeConstrNum{m^+}
\global\long\def\RBCer{\text{RBCer}} 

\title{Proximal Bundle Methods: Global Linear Convergence under Piecewise Smoothness and Universal Acceleration}

\author{Zhenwei Lin\thanks{lin2193@purdue.edu, School of Industrial Engineering, Purdue University} \qquad\qquad\quad Zhe Zhang \thanks{zhan5111@purdue.edu, School of Industrial Engineering, Purdue University}  }

\maketitle

\begin{abstract}
We study why proximal bundle methods (PBMs) can perform better in practice
when they retain more cutting planes and how they can achieve universal
acceleration. The central tool in both analyses is the regularized bundle
certificate (RBCer). 
For convex objectives with quadratic growth and unknown piecewise smooth structure,
we use the RBCer residual as a potential function to prove a global linear rate
for sufficiently many cutting planes.
Our key observation is that retaining sufficiently many cutting planes allows PBM to
exploit the objective’s piecewise-smooth structure and behave as if it were optimizing a smooth function. 
This gives a theoretical explanation for the observed benefit
of more cutting planes on such objectives. 
We then combine a PBM with Nesterov acceleration, using the RBCer to correct for
cutting-plane model. The resulting universal
method attains optimal leading-order rates across H\"older smoothness classes.
\end{abstract}

\section{Introduction\label{sec:introduction}}
We study proximal bundle methods (PBMs) for minimizing a convex function $f$ over
$X$:
{\small
\begin{equation}
f^* = \min_{x\in X} f(x),
\label{eq:problem}
\end{equation}
}
where $X\subseteq\mbb R^n$ is a nonempty closed convex set and
$f:X\to\mbb R$ is lower semicontinuous.
Throughout the paper, we assume $X^*:=\argmin_{x\in X}f(x)\ne\varnothing$.

Developed in the 1970s~\cite{lemarechal1975extension,wolfe1975method,
mifflin1977algorithm}, PBMs are classical methods for nonsmooth convex
optimization, with applications to multicommodity network
flows~\cite{frangioni1999bundle}, Lagrangian duals of nonconvex
problems~\cite{feltenmark2000dual}, and stochastic power
management~\cite{bacaud2001bundle}, among other
areas~\cite{deoliveira2011inexact,ning2020transformation,
kim2022scalable,borndorfer2024electric}.
PBMs stabilize Kelley's cutting-plane model~\cite{kelley1960cutting} by
adding a quadratic proximal term.  At iteration $t$, they compute
{\small
\[
 y^t=\argmin_{x\in X}
 \left\{
  \psi^t(x)
 +\frac{\rho}{2}\norm{x-x^t}^2
 \right\},
\]
}
where $x^t$ is the current proximal center, $\rho>0$ is the proximal
parameter, and $\psi^t\leq f$ is the convex piecewise-affine model formed
by taking the pointwise maximum of the retained cutting planes.
Two questions remain open in the theory of proximal bundle methods.

\underline{Why do PBMs often perform better with more cuts?}
Existing worst-case bounds do not necessarily improve when more cutting
planes are retained. In particular, a full-memory PBM and a two-cut variant
that retains the aggregate plane and the most recent oracle cut can have
worst-case iteration bounds of the same order~\cite{du2017rate,
liang2024unified}.
For general Lipschitz nonsmooth convex problems, suitably
tuned or adaptive variants attain the classical $O(1/\vep^2)$ subgradient
order~\cite{liang2024unified,diaz2023optimal}.
In practice, however, the numerical behavior can be
markedly different.  
In LM-PBM (Algorithm~\ref{alg:lm-pbm}), $B$ is the maximum number of oracle
cuts retained in the piecewise-affine model $\psi^t$.
Figure~\ref{fig:intro-lm-pbm-memory} illustrates the effect of retaining more
cuts on a $\pws$ instance (MAXQUAD).
The plotted gap decays approximately geometrically for $B=50$ and much more
slowly for smaller values of $B$.
Related numerical studies also report that
multi-cuts model can require fewer iterations than more aggressively compressed
models on some nonsmooth problems~\cite{
fersztand2024frankwolfe,guigues2024adaptive}.
\begin{figure}[tbp]
  \centering
  \includegraphics[width=0.65\linewidth]{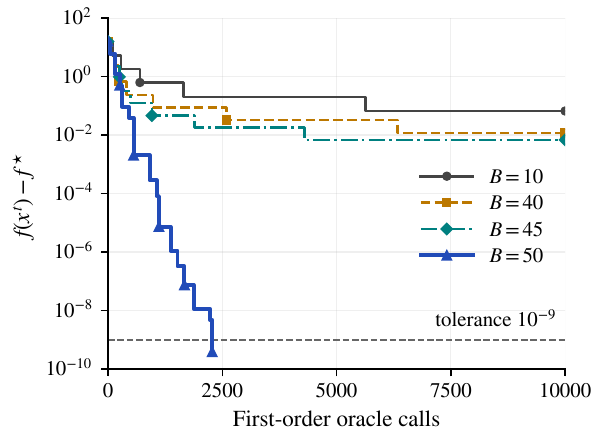}
  \caption[Effect of the bundle size on a random MAXQUAD instance.]{
  Effect of varying the maximum number $B$ of trial cuts appended between
  compressions on a random MAXQUAD instance.
  \label{fig:intro-lm-pbm-memory}}
\end{figure}

% \begin{figure}[tbp]
%   % \setlength{\abovecaptionskip}{0pt}
%   \centering
%   \includegraphics[width=0.6\linewidth]{intro_figure_n1000_k50_tol1e9.pdf}
%   \vspace{-14pt}
%   \caption[Effect of the bundle size on a random MAXQUAD instance.]{
%   Effect of the parameter $B$ in Alg.~\ref{alg:lm-pbm} on a random MAXQUAD instance.
%   \label{fig:intro-lm-pbm-memory}}
% \end{figure}
% Adding cuts can improve the lower model at a fixed center, but this alone
% does not imply a faster decrease in the objective gap.
% Null steps leave the objective value at the center unchanged, serious
% steps move the center, and compression can lower the model by discarding cuts.
% For unknown piecewise-smooth objectives, useful smoothness estimates
% involve queries in the same piece, whose identity is unavailable to the algorithm.
% The challenge is to retain enough information to exploit these estimates
% while controlling the total work across center changes and compression.

\uline{How can PBMs achieve universal acceleration?}
Suppose the convex objective $f$ satisfies the H\"older smoothness
condition~\cite{nesterov2015universal}
{\small
\begin{equation}\label{eq:alm-holder-model}
 0\leq f(y)-\ell_f(y;x)
 \leq\frac{M_\nu}{1+\nu}\norm{y-x}^{1+\nu},\qquad x,y\in X,
\end{equation}
}
for some $\nu\in[0,1]$ and $M_\nu\geq0$, where
$\ell_f(y;x):=f(x)+\inner{f^{\prime}(x)}{y-x}$ and
$f^{\prime}(x)\in\partial f(x)$ is returned by the first-order oracle.
The cases $\nu=0$ and $\nu=1$ cover Lipschitz continuous and smooth
convex objectives, respectively.
We seek a single PBM that attains the optimal rate for each $\nu\in[0,1]$ without knowing $\nu$.
Optimal rates have been established for the two endpoint cases: Lipschitz nonsmooth objectives (\(\nu=0\)) by~\cite{diaz2023optimal}, and objectives with Lipschitz-continuous gradients (\(\nu=1\)) by~\cite{liao2026accelerated}. 
However, the former’s adaptation to smooth objectives does not attain an accelerated rate, whereas the latter’s acceleration guarantee relies on Lipschitz continuity of the gradient. 
For each fixed H\"older exponent $\nu\in[0,1]$, the absolute-error
PBM~\cite{liao2026revisiting} finds an $\vep$-optimal point in
$O(\vep^{-2/(1+\nu)})$ oracle calls without knowing $\nu$. This bound matches
the optimal $O(\vep^{-2/(1+3\nu)})$ dependence on accuracy only in the
nonsmooth case $\nu=0$~\cite{lan2015bundle}.

These observations motivate two questions:
\[\ovalbox{\begin{minipage}{0.95\columnwidth - 1\fboxsep - 0.8pt}%
            \centering \it
Why can retaining more cuts improve PBM performance?

Can a single PBM achieve universal acceleration across H\"older smoothness classes?
        \end{minipage}}
\]

We provide a partial answer to the first question by exploiting the objective's
unknown piecewise smooth structure, and an affirmative answer to the second. 
Both results are based on the regularized bundle certificate
(RBCer, Definition~\ref{def:rbcer}), which is the central tool in our analysis.
Our specific contributions are organized as follows.

1. \textbf{A verifiable certificate for PBMs.}
In Section~\ref{sec:rbcer}, we show how the computable residual of the
$\RBCer$ yields a bound on the optimality gap under quadratic growth.
In Section~\ref{sec:proximal_bundle_algorithm}, we use this residual as a
potential function to analyze null steps for piecewise smooth objectives.
In the accelerated PBMs of
Section~\ref{sec:accelerated-lm-pbm}, the certificate guides lower-bound
updates and adjustments to the proximal parameter.

2. \textbf{Global linear convergence of PBM with periodic aggregation.} 
In Section~\ref{sec:proximal_bundle_algorithm}, we establish global linear
convergence of PBM with periodic aggregation (Alg.~\ref{alg:lm-pbm}).
The analysis uses cuts retained between compression steps to exploit
unknown piecewise smooth structure.
Under Assumption~\ref{ass:first-order-oracle} and the $\mu^*$-QG condition,
the method requires $O\bigl(B(1+L/\mu^*)\log(1/\vep)\bigr)$ first-order
oracle calls to attain an objective gap of at most $\vep$ when $B\geq k$. 
In Section~\ref{sec:accelerated-lm-pbm}, we combine Nesterov acceleration
with restarts guided by RBCer to obtain the accelerated bound
$O\bigl(B(1+\sqrt{L/\mu^*})\log(1/\vep)\bigr)$.

3. \textbf{An anytime, parameter-free PBM with universal acceleration.} 
In Section~\ref{sec:accelerated-lm-pbm}, we develop an accelerated PBM
(Alg.~\ref{alg:adaptive-apbm}) that runs without a target accuracy and adapts
to the unknown H\"older smoothness parameters $(\nu,M_\nu)$ and quadratic
growth modulus $\mu^*$. For any fixed cut budget $B\geq1$ and initial estimate
$\widehat\rho_1\geq\mu^*$, Theorem~\ref{thm:alm-adaptive} bounds the number
of first-order oracle calls needed to find an $\vep$-optimal point.
For each fixed $\nu\in[0,1)$, the leading term is
$O\!\left(1+\left(\frac{M_\nu^2}{(\mu^*)^{1+\nu}\vep^{1-\nu}}\right)^{\frac{1}{1+3\nu}}\right)$.
For $\nu=1$, it is $O\bigl((1+\sqrt{M_1/\mu^*})\log(1/\vep)\bigr)$.
To our knowledge, this is the first PBM to attain the optimal leading rates
across all H\"older smoothness classes under quadratic growth.

\subsection{Related Work}

% Bundle methods extend Kelley's cutting-plane scheme~\cite{kelley1960cutting}.
Bundle methods stabilize the cutting-plane model through proximal
regularization or level constraints.

\paragraph{Proximal bundle methods.}
Within the proximal branch, early work developed adaptive proximity control,
exact-penalty formulations for constrained problems, approximate and
decomposed models, Bregman stabilization, and practical efficiency
\cite{kiwiel1990proximity,kiwiel1991exactpenalty,
kiwiel1995approximations,kiwiel1999bregman,kiwiel2000efficiency}.
General treatments of the resulting framework and its model-management rules
can be found in~\cite{bagirov2014introduction,oliveira2014bundle,
frangioni2020standard}.

Bundle update rules determine which cutting planes to retain individually
and which to combine into an aggregate cutting plane~\cite{kiwiel1983aggregate,frangioni2002generalized,liang2024unified}.
Limited-memory schemes reduce storage for large-scale nonsmooth problems
\cite{haarala2007limited}.  Quantitative analyses range from early efficiency
estimates to increasingly sharp iteration bounds under different proximal
parameters, smoothness classes, and growth conditions
\cite{kiwiel2000efficiency,du2017rate,liang2021proximal,diaz2023optimal,
guigues2024adaptive,liao2026accelerated,fersztand2025acceleration}.

Other extensions address inexact oracles and alternative stabilization
\cite{hintermuller2001approximate,kiwiel2006approximate,de2014convex,
bonnans1995variable,lemarechal1997variable,oliveira2016doubly}, as well as
composite, stochastic, and asynchronous computation
\cite{sagastizabal2013composite,liang2024unified,liang2023proximal,
liang2024singlecut,fischer2025asynchronous}. PBMs have also been extended to
nonconvex, weakly convex, constrained, and difference-of-convex optimization
\cite{kiwiel1996restricted,hare2010redistributed,yang2014constrained,
hare2016inexact,joki2017dc,lv2018constrained,de2019proximal,
atenas2023unified,liang2023proximal,pang2023nonconvex}.
Recent primal--dual interpretations connect bundle subproblems with
conditional-gradient methods~\cite{fersztand2024frankwolfe,
liang2024primaldual}, and aggregated models support computable accuracy
certificates~\cite{burns2026accuracy}.  

\paragraph{Bundle level methods.}
The classical bundle-level method projects onto a level set of the cutting-plane
model and attains the optimal $O(1/\vep^2)$ oracle complexity for Lipschitz
convex minimization on compact convex sets~\cite{lemarechal1995new}.
Variants handle constraints~\cite{kiwiel1995proximal,deng2024uniformly,partTwo},
use limited memory~\cite{ben2005non} or inexact oracles~\cite{de2014level},
and adapt to unknown smoothness~\cite{lan2015bundle,jiang2026optimal}.
For convex objectives with unknown piecewise smoothness, matching trial points in
the same piece yields global linear convergence under quadratic
growth~\cite{zhang2025linearly,partOne}.

\subsection{Outline}

After introducing notation, we state our assumptions on piecewise smoothness and 
quadratic growth in Section~\ref{sec:problem_assumptions}.
Section~\ref{sec:rbcer} introduces the regularized bundle certificate and shows
how its residual bounds the optimality gap under quadratic growth.
We use this certificate in Section~\ref{sec:proximal_bundle_algorithm} to prove
global linear convergence with a sufficiently large bundle and in
Section~\ref{sec:accelerated-lm-pbm} to establish universal acceleration
across H\"older smoothness classes. Section~\ref{sec:pbm-numerical-study}
presents the numerical experiments.

\subsection{Notation}

Unless stated otherwise, $\norm{\cdot}$ and $\inner{\cdot}{\cdot}$ denote the Euclidean norm and inner product.  For a nonempty set $S\subseteq\mbb R^n$, let $\dist(x,S):=\inf_{y\in S}\norm{x-y}$.  For $x\in X$, the convex normal cone is $N_X(x):=\{v\in\mbb R^n:\allowbreak \inner{v}{z-x}\leq0\ \forall z\in X\}$, and the subdifferential is $\partial f(x):=\{v\in\mbb R^n:\allowbreak f(z)\geq f(x)+\inner{v}{z-x}\ \forall z\in X\}$.  For $a\in\mbb R$, write $[a]_+:=\max\{a,0\}$ and $\lceil a\rceil_+:=\max\{0,\lceil a\rceil\}$. For $a\geq0$, write $\log_+(a):=\log(\max\{1,a\})$.

\section{Preliminaries}\label{sec:problem_assumptions}
We introduce the structural and oracle assumptions used in the analysis.
\begin{defn}[Piecewise smoothness]\label{def:pws}
Let $k\in\mathbb N_+$ and $L>0$. A function $f:X\to\mbb R$ is
$(k,L)$-piecewise smooth (PWS) if $X$ has a cover $X_1,\ldots,X_k$
such that, for every $i\in[k]$ and $\bar x\in X_i$, there exists
$v\in\mbb R^n$ satisfying $|f(x)-f(\bar x)-\inner{v}{x-\bar x}|
\leq(L/2)\norm{x-\bar x}^2$ for all $x\in X_i$.
We call $X_1,\ldots,X_k$ the smooth pieces and $L$ the
piecewise-smoothness constant.
\end{defn}
\begin{assumption}\label{ass:first-order-oracle}
The function $f$ is $(k,L)$-PWS with a fixed cover
$X_1,\ldots,X_k$ as in Definition~\ref{def:pws}.
At every query point $\bar x\in X$, the first-order oracle returns
$f(\bar x)$ and a subgradient $f^{\prime}(\bar x)\in\partial f(\bar x)$ satisfying
{\small
\begin{equation}\label{eq:pws}
 0\leq f(x)-\ell_f(x;\bar x)\leq\frac L2\norm{x-\bar x}^2,
 \qquad i\in[k],\quad x,\bar x\in X_i.
\end{equation}
}
\end{assumption}

We give four remarks on Definition~\ref{def:pws} and
Assumption~\ref{ass:first-order-oracle}.

First, the finite maximum $f(x)=\max_{i\in[k]}f_i(x)$ of convex $L$-smooth
functions $f_i:\mbb R^n\to\mbb R$ satisfies
Assumption~\ref{ass:first-order-oracle} with an active-component oracle.
Assign each point to an active component by a fixed tie-breaking rule,
let $X_i$ consist of the points assigned to $i$, and return
$f^{\prime}(\bar x)=\nabla f_i(\bar x)$ for $\bar x\in X_i$.
Since $f\geq f_i$ with equality on $X_i$, we have
$f^{\prime}(\bar x)\in\partial f(\bar x)$, and the convexity and smoothness of
$f_i$ give~\eqref{eq:pws}.
Second, piecewise smoothness alone does not guarantee~\eqref{eq:pws} for an
arbitrary subgradient. For example, $f(x)=|x|$ is PWS with pieces
$(-\infty,0)$ and $[0,\infty)$, but the valid choice $f^{\prime}(0)=0$
violates~\eqref{eq:pws}: for any $L>0$,
$f(x)-\ell_f(x;0)=x>(L/2)x^2$ whenever $0<x<2/L$.
Third, for the finite maximum and cover in the first remark,
suppose $\operatorname{int}X\ne\varnothing$. At every differentiability
point in $\operatorname{int}X$, each subgradient equals the gradient of
every active component, so~\eqref{eq:pws} holds automatically.
Possible violations are confined to nondifferentiability points and
$\partial X$. These sets have Lebesgue measure zero but may contain
iterates of the algorithm, so an almost-everywhere guarantee is insufficient.
Fourth, the smooth pieces may be nonconvex or lower dimensional.
Neither the pieces nor the parameters $k$ and $L$ need be known
to the algorithm.

Furthermore, we impose the following quadratic-growth (QG) condition.

\begin{assumption}[Quadratic growth]\label{def:QG}
The function $f$ satisfies the QG condition with maximal modulus
$\mu^*:=\sup\left\{\nu\geq0:
 f(x)-f^*\geq\frac\nu2\dist^2(x,X^*),\forall x\in X\right\}>0.$
\end{assumption}

\section{Regularized bundle certificate}\label{sec:rbcer}
We define a regularized bundle certificate for a general convex minorant
using a quantity from~\cite[Section~5.3]{diaz2023optimal}.
We show that this certificate bounds the objective gap under the QG condition.

\begin{defn}[Regularized bundle certificate, RBCer]\label{def:rbcer}
Let $\rho>0$ and let $\psi:X\to\mbb R$ be lower semicontinuous and
convex with $\psi\leq f$ on $X$.  For $\bar x\in X$ and $\eta\geq0$,
the pair $(\bar x,\psi)$ is a $(\rho,\eta)$-regularized bundle
certificate if $\Delta_\rho(\bar x;\psi)\leq\eta$, where
{\small
\begin{equation}\label{eq:rbcer}
 \Delta_\rho(\bar x;\psi):=\max_{x\in X}
 \left\{f(\bar x)-\psi(x)-\frac\rho2\norm{x-\bar x}^2\right\}.
\end{equation}
}
\end{defn}
An affine lower bound on $\psi$ makes the regularized objective
coercive. Lower semicontinuity and strong convexity give a unique
minimizer.  Evaluating at $\bar x$ shows $\Delta_\rho(\bar x;\psi)\geq0$.
Here the quadratic term regularizes only the model subproblem, and
neither boundedness of $X$ nor center exactness is required.

The RBCer extends the gradient-norm criterion.
Indeed, when $X=\mbb R^n$ and $f$ is differentiable at $\bar x$,
completing the square for the tangent model gives
{\small
\[
 \Delta_\rho(\bar x;\ell_f(\cdot;\bar x))
 =\max_{x\in\mbb R^n}\left\{
 -\inner{\nabla f(\bar x)}{x-\bar x}
 -\frac\rho2\norm{x-\bar x}^2\right\}
 =\frac{\norm{\nabla f(\bar x)}^2}{2\rho}.
\]
}
This interpretation parallels the normalized Wolfe certificate
in~\cite{partOne,partTwo,jiang2026optimal}. The normalized Wolfe
certificate measures the maximal decrease of a
cutting-plane model over a ball centered at $\bar x$ and intersected
with $X$, normalized by the radius.
The RBCer searches over all of $X$ and subtracts
$(\rho/2)\norm{x-\bar x}^2$ from the model decrease.
A point farther from $\bar x$ must therefore offer more model decrease
to offset the larger penalty.
With $\bar x$ and $\rho$ unchanged, adding cuts while retaining the
existing ones cannot increase $\Delta_\rho$. 
Thus, $\sqrt{2\rho\Delta_\rho}$ naturally extends the gradient-norm
criterion to nonsmooth optimization.

Under the QG condition, the RBCer yields the following upper bound on the objective gap.

\begin{proposition}\label{prop:rbcer-objective-gap}
If $f$ satisfies the QG condition, $0<\widehat\mu\leq\mu^*$, and $(\bar x,\psi)$
is a $(\rho,\eta)$-RBCer, then $f(\bar x)-f^*\leq\max\{2,4\rho/\widehat\mu\}\eta$.
\end{proposition}
\begin{proof}
Choose $\bar x^*\in X^*$ with
$\norm{\bar x-\bar x^*}=\dist(\bar x,X^*)$, and let
$z_\lambda=(1-\lambda)\bar x+\lambda\bar x^*$ for $\lambda\in[0,1]$.
Convexity gives $z_\lambda\in X$ and
$f(z_\lambda)\leq(1-\lambda)f(\bar x)+\lambda f^*$, while
the QG condition and $\widehat\mu\leq\mu^*$ imply
$\norm{\bar x-\bar x^*}^2\leq2(f(\bar x)-f^*)/\widehat\mu$.
Evaluating the maximum in~\eqref{eq:rbcer} at $z_\lambda$ and using
$\psi(z_\lambda)\leq f(z_\lambda)$ gives
{\small
\begin{equation}\label{eq:lambda-bound}
 \Delta_\rho(\bar x;\psi) \geq f(\bar x)-f(z_\lambda)-\frac\rho2\norm{z_\lambda-\bar x}^2\geq\left(\lambda-\frac\rho{\widehat\mu}\lambda^2\right)
 (f(\bar x)-f^*).
\end{equation}
}
The second inequality uses the preceding convexity and QG bounds together
with $z_\lambda-\bar x=\lambda(\bar x^*-\bar x)$.

Since $\Delta_\rho(\bar x;\psi)\leq\eta$, choosing $\lambda=1$ when
$\rho\leq\widehat\mu/2$ gives $f(\bar x)-f^*\leq2\eta$.
Otherwise, choosing $\lambda=\widehat\mu/(2\rho)\in(0,1)$ gives
$f(\bar x)-f^*\leq4\rho\eta/\widehat\mu$, completing the proof.
\end{proof}
In particular,
$f(\bar x)-\max\{2,4\rho/\widehat\mu\}\eta\leq f^*,0<\widehat\mu\leq\mu^*.$
A feasible objective value below this conditional lower bound therefore
certifies $\widehat\mu>\mu^*$. 

% By the subgradient inequality and completing the square in~\eqref{eq:rbcer},
% a single oracle query at any $u\in X$ yields a
% $(\rho,\norm{f^{\prime}(u)}^2/(2\rho))$-RBCer $(u,\ell_f(\cdot;u))$
% for every $\rho>0$.

\section{Proximal Bundle Methods}
\label{sec:proximal_bundle_algorithm}\label{sec:limited_memory_algorithm}
In this section, we show how PBMs with a fixed proximal parameter exploit piecewise smooth structure.

\subsection{A proximal bundle method with periodic aggregation}\label{subsec:lm-algorithm}
At center $x^t$, the model $\psi^t$ is a finite maximum of affine
functions satisfying $\psi^t\leq f$ on $X$.
Alg.~\ref{alg:lm-pbm} computes an
exact minimizer $y^t$ of $\psi^t(x)+\rho\norm{x-x^t}^2/2$ over $X$
and queries the oracle at $y^t$.  If the sufficient-decrease test holds,
$y^t$ becomes the next center (a serious step).  Otherwise, the center
remains unchanged (a null step).

Alg.~\ref{alg:lm-pbm} starts from a single oracle cut with $b=1$.
The counter $b$ records the number of explicitly retained oracle cuts,
excluding the aggregate plane.
At a null step with $b<B$, line~\ref{line:lm-append} of
Alg.~\ref{alg:lm-pbm} appends $\ell_f(\cdot;y^t)$ and increments $b$.
A serious step or $b=B$ instead triggers line~\ref{line:lm-compress}
of Alg.~\ref{alg:lm-pbm}, which replaces the model by
$\max\{a^t,\ell_f(\cdot;y^t)\}$ and resets $b=1$.
Here $a^t$ is the aggregate plane
\cite{kiwiel1983aggregate,diaz2023optimal,liang2024unified}:
{\small
\begin{equation}\label{eq:aggregate}
 a^t(x):=\psi^t(y^t)+\rho\inner{x^t-y^t}{x-y^t}.
\end{equation}
}
Lemma~\ref{lem:aggregate} shows that $a^t\leq\psi^t$ on $X$.
After initialization or compression, $B-1$ consecutive null steps
append cuts, and the $B$th null step compresses the model.
Thus, each model contains at most $B$ oracle cuts and one aggregate plane,
for a total of at most $B+1$ affine functions.

\begin{algorithm}[!htbp]
\caption{Periodic Aggregated Proximal Bundle Method}
\label{alg:lm-pbm}
\begin{algorithmic}[1]
\Require $x^0\in X$, $\rho>0$, $\beta\in(0,1)$, and $B\in\mathbb N_+$
\State $\psi^0\gets\ell_f(\cdot;x^0)$ and $b\gets1$
\label{line:lm-initialize}
\For{$t=0,1,2,\ldots$}
  \State $\displaystyle y^t\gets\argmin_{x\in X}
  \left\{\psi^t(x)+\frac{\rho}{2}\norm{x-x^t}^2\right\}$, $\Delta_t\gets f(x^t)-\psi^t(y^t)-\frac{\rho}{2}\norm{x^t-y^t}^2$
  \If{$f(x^t)-f(y^t)\geq
  \beta\bigl(f(x^t)-\psi^t(y^t)\bigr)$}
    \State $x^{t+1}\gets y^t$ \Comment{serious step}
  \Else
    \State $x^{t+1}\gets x^t$ \Comment{null step}
  \EndIf
  \If{iteration $t$ is serious or $b=B$}
    \State $\psi^{t+1}(x)\gets
    \max\{a^t(x),\ell_f(x;y^t)\}$, $b\gets1$
    \Comment{$a^t$ is defined in~\eqref{eq:aggregate}}
    \label{line:lm-compress}
  \Else
    \State $\psi^{t+1}(x)\gets
    \max\{\psi^t(x),\ell_f(x;y^t)\}$, $b\gets b+1$
    \label{line:lm-append}
  \EndIf
\EndFor
\end{algorithmic}
\end{algorithm}

\begin{rem}\label{rem:lm-periodic-aggregation}
The two-plane PBM update described in~\cite[(2.7)]{diaz2023optimal}
sets $\psi^{t+1}=\max\{a^t,\ell_f(\cdot;y^t)\}$ at every iteration.
Alg.~\ref{alg:lm-pbm} uses the same aggregate plane but postpones
compression during null steps until $b=B$.
For the query pairs $l<r$ in Proposition~\ref{prop:lm-transition},
the intervening cut additions ensure both
$\psi^r\geq\ell_f(\cdot;y^l)$ and $\psi^r\geq\psi^{l+1}$.
These inequalities allow a same-piece pair among $B+1$ queries to yield
a certificate contraction depending on $L$ when $B\geq k$.
Compression at every iteration preserves the single-step estimates but
need not preserve these comparisons for nonadjacent queries, so the
present PWS complexity proof does not directly apply to that update.
For $k=1$, adjacent queries suffice and $B=1$ is admissible.
Our analysis establishes a sufficient compression schedule, without
asserting its necessity for linear convergence.
\end{rem}

Alg.~\ref{alg:lm-pbm} computes $\Delta_t$ at each iteration, and
Proposition~\ref{prop:rbcer-objective-gap} converts it into an upper bound
on the objective gap under QG.  To bound the oracle complexity, the
following proposition controls the possible increase of $\Delta_t$ at
serious steps and quantifies its decrease over consecutive null steps
in terms of the memory parameter $B$.
\begin{proposition}\label{prop:lm-transition}
Let $\gamma_\rho:=1-\beta/\max\{2,4\rho/\mu^*\}$.
If iteration $t$ of Alg.~\ref{alg:lm-pbm} is serious, then
$\Delta_{t+1}\leq(3-2\beta)\Delta_t$ and
$f(x^{t+1})-f^*\leq\gamma_\rho(f(x^t)-f^*)$.
If iteration $t$ is null, then $\Delta_{t+1}\leq\Delta_t$.
For $B+1$ consecutive null iterations $u,\ldots,u+B$ with
$b=B$ at the start of iteration $u$, define
{\small
\begin{equation}\label{eq:lm-empirical-smoothness}
 L(B):=\min_{u\leq l<r\leq u+B}L(l,r),\qquad
 L(l,r):=\frac{2\bigl(f(y^r)-\ell_f(y^r;y^l)\bigr)}{\norm{y^r-y^l}^2}.
\end{equation}
}
Then
{\small
\begin{equation}\label{eq:lm-memory-contraction}
 \Delta_{u+B}\leq\theta_\rho(B)\Delta_u,\qquad
 \theta_\rho(B):=\frac{L(B)}{L(B)+(1-\beta)\rho}.
\end{equation}
}
Furthermore, if $B\geq k$, then $L(B)\leq L$. 
\end{proposition}
For $B\geq k$, Proposition~\ref{prop:lm-transition} gives a certificate
contraction factor of at most $L/(L+(1-\beta)\rho)$ over $B+1$
consecutive null steps starting with $b=B$.  Combining this decrease
with the objective-gap contraction at serious steps yields the following
global oracle complexity bound.
\begin{thm}
\label{thm:lm-rbcer-complexity}\label{thm:lm-objective-gap}
Suppose Assumption~\ref{ass:first-order-oracle} holds and $f$ is
$\mu^*$-QG.
Let $B\geq k$, $\beta\in(0,1)$, and
$\rho\in[\min\{\mu^*,L\},\allowbreak\max\{\mu^*,L\}]$.
For any $\vep>0$, set $\eta=\vep/\max\{2,4\rho/\mu^*\}$.
Alg.~\ref{alg:lm-pbm} reaches the first iteration $t$ satisfying
$\Delta_t\leq\eta$ within
{\small
\begin{equation}\label{eq:LM-bigO}
 O\!\left(B\left(1+\frac L{\mu^*}\right)
 \left[1+\log_+\frac{2\norm{f^{\prime}(x^0)}^2}{\mu^*\vep}
 +\log\max\left\{\frac\rho{\mu^*},\frac{\mu^*}\rho\right\}
 \right]\right)
\end{equation}
}
oracle calls.
At this iteration, $(x^t,\psi^t)$ is a $(\rho,\eta)$-RBCer and
$f(x^t)-f^*\leq\vep$.
\end{thm}
The global condition $B\geq k$ is sufficient for the complexity
guarantee in Theorem~\ref{thm:lm-rbcer-complexity}, but $B$ may be smaller than the global
piece count $k$.  The same guarantees hold whenever the trial points
in each block of $B+1$ consecutive null steps starting with $b=B$
can be covered by at most $B$ smooth pieces.

\subsection{Proof details}
\label{subsec:limited-memory-detailed-proofs}
We first establish lower bounds on the updated regularized model in
Lemma~\ref{lem:aggregate}.  These bounds yield the serious- and null-step
estimates in Proposition~\ref{prop:lm-transition}, including the dependence
of null-step progress on $B$.  We then combine these estimates to prove
the oracle complexity in Theorem~\ref{thm:lm-rbcer-complexity}.

\begin{lem}\label{lem:aggregate}
For every $x\in X$, $a^t(x)\leq\psi^t(x)$ and
$\psi^{t+1}(x)\geq a^t(x)$.  Moreover,
{\small
\begin{align}
 \psi^{t+1}(x)+\frac\rho2\norm{x-x^t}^2
 &\geq\psi^t(y^t)+\frac\rho2\norm{y^t-x^t}^2
   +\frac\rho2\norm{x-y^t}^2,\label{eq:key-proximal-bridge}\\
 \psi^{t+1}(x)+\frac\rho2\norm{x-y^t}^2
 &\geq\psi^t(y^t)-\frac\rho2\norm{y^t-x^t}^2.
 \label{eq:agg-identity-new-center}
\end{align}
}
\end{lem}
\begin{proof}
The finite affine model is continuous, so constrained proximal optimality
gives $\rho(x^t-y^t)=v^t+n^t$ with
$v^t\in\partial\psi^t(y^t)$ and $n^t\in N_X(y^t)$.  The subgradient
inequality and $\inner{n^t}{x-y^t}\leq0$ imply, for every $x\in X$,
{\small
\begin{equation*}
 \psi^t(x)\geq\psi^t(y^t)+\inner{v^t}{x-y^t}
 \geq\psi^t(y^t)+\rho\inner{x^t-y^t}{x-y^t}=a^t(x).
\end{equation*}
}
Both model updates preserve $\psi^{t+1}\geq a^t$ on $X$:
compression in line~\ref{line:lm-compress} of Alg.~\ref{alg:lm-pbm} sets
$\psi^{t+1}=\max\{a^t,\ell_f(\cdot;y^t)\}\geq a^t$, while
cut addition in line~\ref{line:lm-append} of Alg.~\ref{alg:lm-pbm} gives
$\psi^{t+1}\geq\psi^t\geq a^t$.

Adding the proximal term $\rho\norm{x-x^t}^2/2$
to~\eqref{eq:aggregate} and completing the square yields
{\small
\begin{equation*}
 a^t(x)+\frac\rho2\norm{x-x^t}^2
 =\psi^t(y^t)+\frac\rho2\norm{y^t-x^t}^2
  +\frac\rho2\norm{x-y^t}^2.
\end{equation*}
}
Since $\psi^{t+1}\geq a^t$, this proves~\eqref{eq:key-proximal-bridge}.
For the bound at the new center, substituting~\eqref{eq:aggregate} gives
{\small
\begin{align*}
 a^t(x)+\frac\rho2\norm{x-y^t}^2
 &=\psi^t(y^t)+\rho\inner{x^t-y^t}{x-y^t}
   +\frac\rho2\norm{x-y^t}^2\\
 &=\psi^t(y^t)-\frac\rho2\norm{x^t-y^t}^2
   +\frac\rho2\norm{x+x^t-2y^t}^2.
\end{align*}
}
The second equality follows by expanding
$\norm{(x-y^t)+(x^t-y^t)}^2$.
Dropping the final nonnegative term and using $\psi^{t+1}\geq a^t$
proves~\eqref{eq:agg-identity-new-center}.
\end{proof}

\begin{proof}[\underline{Proof of Proposition~\ref{prop:lm-transition}}]
Strong convexity of the proximal objective and $\psi^t\leq f$ give
$f(x^t)\geq\psi^t(x^t)\geq\psi^t(y^t)+\rho\norm{y^t-x^t}^2$.
The definition of $\Delta_t$ therefore implies
{\small
\begin{equation}\label{eq:proximal-bounds}
 \frac\rho2\norm{y^t-x^t}^2\leq\Delta_t,\qquad
 \Delta_t\leq f(x^t)-\psi^t(y^t)\leq2\Delta_t.
\end{equation}
}
At a serious iteration, Alg.~\ref{alg:lm-pbm} sets
$x^{t+1}=y^t$.  The serious-step test and
Proposition~\ref{prop:rbcer-objective-gap} yield
$f(x^{t+1})-f^*
 \leq f(x^t)-f^*-\beta\Delta_t \leq\gamma_\rho\bigl(f(x^t)-f^*\bigr).$

For the new certificate, we have
{\small
\begin{align*}
 \Delta_{t+1}
 &=f(y^t)-\min_{x\in X}
   \left\{\psi^{t+1}(x)+\frac\rho2\norm{x-y^t}^2\right\}{\leq} f(y^t)-\psi^t(y^t)+\frac\rho2\norm{y^t-x^t}^2\\
 &\leq(1-\beta)\Delta_t+(2-\beta)\frac\rho2\norm{y^t-x^t}^2
 \leq(3-2\beta)\Delta_t.
\end{align*}
}
The first inequality holds by~\eqref{eq:agg-identity-new-center} and the last two inequalities follow from the serious-step test and~\eqref{eq:proximal-bounds}.

At a null iteration, $x^{t+1}=x^t$, so evaluating
\eqref{eq:key-proximal-bridge} at $y^{t+1}$ yields
{\small
\begin{equation}\label{eq:lm-null-decrease}
 \Delta_t-\Delta_{t+1}\geq\frac\rho2\norm{y^{t+1}-y^t}^2\geq0. 
\end{equation}
}

Now consider the stated null iterations.  Iteration $u$ compresses the
model to $\max\{a^u,\ell_f(\cdot;y^u)\}$ and resets $b=1$.
Iterations $u+1,\ldots,u+B-1$ append cuts, so the model at iteration
$u+B$ retains the cuts at $y^u,\ldots,y^{u+B-1}$.
All query points use the same center $x^u$.
For any $u\leq l<r\leq u+B$, the update at iteration $l$ retains
$\ell_f(\cdot;y^l)$, and no compression occurs at iterations
$l+1,\ldots,r-1$.  Hence
$\psi^r\geq\psi^{l+1}$ and $\psi^r\geq\ell_f(\cdot;y^l)$.
A repeated query $y^r=y^l$ would make this retained cut exact at $y^r$,
contradicting the null-step inequality
$f(y^r)-\psi^r(y^r)>(1-\beta)\Delta_r\geq0$.
Thus the query points are distinct, and $L(l,r)>0$ is well defined.
Using $\psi^r\geq\psi^{l+1}$ and applying
\eqref{eq:key-proximal-bridge} at $x=y^r$ gives
{\small
\begin{equation*}
 \Delta_l-\Delta_r
 \geq\psi^{l+1}(y^r)-\psi^l(y^l)
 +\frac\rho2\bigl(\norm{y^r-x^u}^2-\norm{y^l-x^u}^2\bigr)
 \geq\frac\rho2\norm{y^r-y^l}^2.
\end{equation*}
}
The null-step test at $r$, the retained cut at $y^l$, and
\eqref{eq:lm-empirical-smoothness} imply
{\small
\begin{equation*}
 (1-\beta)\Delta_r
 < f(y^r)-\psi^r(y^r)
 \leq\frac{L(l,r)}2\norm{y^r-y^l}^2
 \leq\frac{L(l,r)}\rho(\Delta_l-\Delta_r).
\end{equation*}
}
Rearranging and using~\eqref{eq:lm-null-decrease} gives
$\Delta_{u+B}\leq\Delta_r
\leq L(l,r)\Delta_u/(L(l,r)+(1-\beta)\rho)$.
Taking the minimum over all pairs proves~\eqref{eq:lm-memory-contraction}.
If $B\geq k$, two of the $B+1$ query points lie in the same smooth
piece, and Assumption~\ref{ass:first-order-oracle} gives $L(l,r)\leq L$
for that pair.  Hence $L(B)\leq L$.  The same argument applies when
the block's query points can be covered by at most $B$ smooth pieces.
\end{proof}

\begin{proof}[\underline{Proof of Theorem~\ref{thm:lm-rbcer-complexity}}]
The implicit constants below may depend on the fixed $\beta\in(0,1)$.
If $\Delta_0\leq\eta$, the initial oracle query suffices.
Otherwise, consider any integer $T\geq0$ such that $\Delta_t>\eta$ for
$0\leq t\leq T$. 
We first bound the number of serious steps, use this bound to control
the number of complete blocks of null steps, and then count the oracle calls.

 Let $S$ be the number of serious iterations
among $0,\ldots,T-1$.
For $S\geq1$, Proposition~\ref{prop:lm-transition} bounds the objective
gap before the last serious step by $\gamma_\rho^{S-1}(f(x^0)-f^*)$,
since null steps leave the center unchanged.  That step decreases $f$
by more than $\beta\eta$, so
$\beta\eta<\gamma_\rho^{S-1}(f(x^0)-f^*)$.
The following bound also holds for $S=0$:
{\small
\begin{equation*}
 S\leq1+\frac{\log_+((f(x^0)-f^*)/(\beta\eta))}{\log(1/\gamma_\rho)}
 =O\!\left(\left(1+\frac\rho{\mu^*}\right)
 \left[1+\log_+\frac{f(x^0)-f^*}{\eta}\right]\right).
\end{equation*}
}

Within these $T$ iterations, consider each consecutive sequence of null
iterations.  The counter starts at $b=1$, so the first compression occurs
at the $B$th iteration of the sequence, and subsequent compressions are
separated by $B$ iterations.
Let $P$ count pairs of consecutive compressions within the same null sequence.
For each such pair $u,u+B$, Proposition~\ref{prop:lm-transition} and
$B\geq k$ give
$\Delta_{u+B}\leq [L/(L+(1-\beta)\rho)]\Delta_u$.
The same proposition bounds the growth factor at each serious step by
$3-2\beta$. The certificate is nonincreasing at the remaining null steps.
Consequently,
$\eta<\Delta_T\leq\Delta_0(3-2\beta)^S
 (L/(L+(1-\beta)\rho))^P$, which gives
{\small
\begin{equation*}
 P\leq1+\frac{S\log(3-2\beta)+\log_+(\Delta_0/\eta)}{\log(1+(1-\beta)\rho/L)}
 =O\!\left(\left(1+\frac L\rho\right)
 [1+S+\log_+(\Delta_0/\eta)]\right).
\end{equation*}
}
Here we use $-\log(1-v)\geq v$ for $0<v<1$ and
$\log(1+v)\geq v/(1+v)$ for $v>0$.

The $S$ serious iterations separate at most $S+1$ null sequences.
Each sequence contains fewer than $B$ iterations before its first
compression and after its last. A sequence without compression also
has fewer than $B$ iterations.
Consequently, $T\leq S+BP+2B(S+1)$.
Since $\psi^0=\ell_f(\cdot;x^0)$, completing the square in~\eqref{eq:rbcer}
gives $\Delta_0\leq\norm{f^{\prime}(x^0)}^2/(2\rho)$.
Combining these estimates gives a uniform bound on $T$ of
{\small
\begin{equation}\label{eq:LM-computable-stop}
 O\!\left(
 B\left(1+\frac L\rho\right)
 \left[\left(1+\frac\rho{\mu^*}\right)
 \left(1+\log_+\frac{f(x^0)-f^*}{\eta}\right)
 +\log_+\frac{\norm{f^{\prime}(x^0)}^2}{\rho\eta}\right]\right).
\end{equation}
}
Thus, an iteration with $\Delta_t\leq\eta$ must occur.
Each iteration uses one oracle query, so initialization and the
final iteration affect only the constant in this bound.
At the first such iteration, Proposition~\ref{prop:rbcer-objective-gap}
gives the stated objective guarantee.

The convexity and the QG condition imply
$f(x^0)-f^*\leq2\norm{f^{\prime}(x^0)}^2/\mu^*$.
With this choice of $\eta$, both logarithmic terms
in~\eqref{eq:LM-computable-stop} are therefore at most
$
 \log 4+\log_+\frac{2\norm{f^{\prime}(x^0)}^2}{\mu^*\vep}
 +\log\max\left\{\frac\rho{\mu^*},\frac{\mu^*}\rho\right\}.
$
Moreover, the assumed interval gives $(\rho-L)(\rho-\mu^*)\leq0$,
hence $L/\rho+\rho/\mu^*\leq1+L/\mu^*$ and
$(1+L/\rho)(1+\rho/\mu^*)\leq2(1+L/\mu^*)$.
Substituting these estimates into~\eqref{eq:LM-computable-stop}
and absorbing fixed constants yields~\eqref{eq:LM-bigO}.
\end{proof}

\section{Accelerated Proximal Bundle Methods}
\label{sec:accelerated-lm-pbm}
Using the certificate from Section~\ref{sec:rbcer}, we develop a universal
accelerated PBM and give its PWS and H\"older complexities. These bounds
guide the parameter search in Subsection~\ref{subsec:alm-adaptive}.
All proofs appear in Subsection~\ref{subsec:alm-proofs}.

\subsection{A universal accelerated proximal bundle method}
\label{subsec:alm-algorithm}
The method combines the APBM-One-Step subroutine in
Alg.~\ref{alg:alm-block} with the stagewise bound updates in
Alg.~\ref{alg:accelerated-lm-pbm}.
Within each stage, fix a reference $\bar{y}$, a level $\ell<f(\bar{y})$,
and a proximal parameter $\rho>0$.  Starting from $\hat{x}^{1}=\bar{y}$,
line~\ref{line:alm-initialize} of Alg.~\ref{alg:alm-block}
sets $\alpha_t=2/(t+1)$.
Lines~\ref{line:alm-proximal} and~\ref{line:alm-interpolate} of Alg.~\ref{alg:alm-block} compute
the proximal minimizer $x^{t,i}$ and form the query point
$(1-\alpha_t)\hat{x}^{t}+\alpha_t x^{t,i}$, respectively.
The point $\hat{x}^{t}$ stays fixed during a call, while
line~\ref{line:alm-best-point} of Alg.~\ref{alg:alm-block} updates $\hat{x}^{t+1}$ to the best point
available.

The model retains an affine minorant $a^{t-1}\leq f$ with weight
$\lambda_{t-1}\geq0$.  Line~\ref{line:alm-initialize} of Alg.~\ref{alg:alm-block} initializes
the local bundle as
$\psi^{t,0}=\max\{\ell,\ell_f(\cdot;\hat{x}^{t})\}$.
We retain the oracle information at each incumbent, so the initial cut
requires no additional query.  Line~\ref{line:alm-add-cut} of
Alg.~\ref{alg:alm-block} adds $B$ further oracle cuts over the call.
The local bundle has weight $t$ in the subproblem, and the incoming
weighted aggregate is retained throughout the call.
After the last minimization, line~\ref{line:alm-aggregate} of Alg.~\ref{alg:alm-block}
compresses all oracle cuts in the local bundle into the aggregate using
their dual multipliers, which is different from the previous
aggregation~\eqref{eq:aggregate}.
The final local model is $\psi^{t,B}=\max\{\ell,c_0,c_1,\ldots,c_B\}$, where
$c_0=\ell_f(\cdot;\hat{x}^{t})$ and
$c_j=\ell_f(\cdot;\bar{x}^{t,j-1})$ for $j=1,\ldots,B$.
In an optimal dual solution of the last subproblem in
line~\ref{line:alm-proximal} of Alg.~\ref{alg:alm-block}, let
$\gamma_j\geq0$ be the epigraph multiplier associated with cut $c_j$
for $j=0,\ldots,B$.  The multiplier associated with the constant
plane $\ell$ is $\gamma_\ell\geq0$.
The optimality conditions of the proximal subproblem give
$\sum_{j=0}^B\gamma_j+\gamma_\ell=t$.
Update the weight and aggregate plane by
{\small
\begin{equation}\label{eq:alm-dual-aggregate}
 \lambda_t=\lambda_{t-1}+\sum_{j=0}^B\gamma_j,\qquad
 a^t=
 \begin{cases}
  \dfrac{\lambda_{t-1}a^{t-1}+\sum_{j=0}^B\gamma_jc_j}{\lambda_t},
  &\lambda_t>0,\\[1ex]
  \ell_f(\cdot;\bar{x}^{t,0}),&\lambda_t=0.
 \end{cases}
\end{equation}
}

As in Section~\ref{sec:proximal_bundle_algorithm}, $a^t$ denotes the
retained affine minorant, here formed from dual multipliers.
Both the weighted plane $\lambda_t(a^t-\ell)$ and the
aggregate~\eqref{eq:aggregate} applied to the last shifted model
preserve its proximal minimum.  They coincide when $X=\mathbb R^n$.
For general $X$, the construction~\eqref{eq:aggregate} includes an
additional normal-cone term, while $a^t$ remains a convex combination
of affine minorants of $f$.

\begin{algorithm}[!htbp]
\caption{APBM-One-Step$(\bar{y},\hat{x}^{t},\ell,\rho,B,t,\lambda_{t-1},a^{t-1})$}
\label{alg:alm-block}
\begin{algorithmic}[1]
\State $\alpha_t\gets2/(t+1)$,
$\hat{x}^{t+1}\gets \hat{x}^{t}$, $\psi^{t,0}\gets\max\{\ell,\ell_f(\cdot;\hat{x}^{t})\}$
\label{line:alm-initialize}
\For{$i=0,\ldots,B$}
 \State $\displaystyle x^{t,i}\gets\argmin_{x\in X}
 \left\{\lambda_{t-1}(a^{t-1}(x)-\ell)
       +t(\psi^{t,i}(x)-\ell)+\frac\rho2\norm{x-\bar{y}}^2\right\}$
 \label{line:alm-proximal}
 \State $\bar{x}^{t,i}\gets(1-\alpha_t)\hat{x}^{t}+\alpha_t x^{t,i}$
 \label{line:alm-interpolate}
 \State \textbf{if} $f(\bar{x}^{t,i})<f(\hat{x}^{t+1})$,
 \textbf{then} $\hat{x}^{t+1}\gets\bar{x}^{t,i}$
 \label{line:alm-best-point}
 \State \textbf{if} $i<B$, \textbf{then}
 $\psi^{t,i+1}\gets\max\{\psi^{t,i},\ell_f(\cdot;\bar{x}^{t,i})\}$
 \label{line:alm-add-cut}
 \textbf{else} update $(\lambda_t,a^t)$ by~\eqref{eq:alm-dual-aggregate}
 \label{line:alm-aggregate}
\EndFor
\State \textbf{Return} $(\hat{x}^{t+1},x^{t,B},\lambda_t,a^t)$
\end{algorithmic}
\end{algorithm}

\begin{algorithm}[!htbp]
\caption{Accelerated Proximal Bundle Methods (APBMs)}
\label{alg:accelerated-lm-pbm}
\begin{algorithmic}[1]
\Require $\bar{y}^{1}\in X$, $0<\rho\leq\mu^*$, $B\in\mathbb N_+$,
$\beta\in(1/2,1)$
\State \textbf{If} $\norm{f^{\prime}\brbra{\bar{y}^{1}}}=0$, \textbf{then} \textbf{return} $\bar{y}^{1}$ 
\State $a\gets\ell_f(\cdot;\bar{y}^{1})$,
$\underline f\gets f(\bar{y}^{1})-2\norm{f^{\prime}(\bar{y}^{1})}^2/\rho$
\For{$s=1,2,\ldots$}
 \State $\Delta_s\gets f(\bar{y}^{s})-\underline f$,
 $\ell_s\gets f(\bar{y}^{s})-\beta\Delta_s$,
 $(\hat{x},\lambda,t)\gets(\bar{y}^{s},0,0)$
 \While{\textbf{true}}
  \State $t\gets t+1$;
  $(\hat{x},x,\lambda,a)\gets\text{APBM-One-Step}(\bar{y}^{s},\hat{x},\ell_s,\rho,B,t,\lambda,a)$
  \State $\displaystyle\Delta_{\rho/\lambda}(\bar{y}^{s};a)\gets
  f(\bar{y}^{s})-a(x)-\frac{\rho}{2\lambda}\norm{x-\bar{y}^{s}}^2$
  \label{line:alm-evaluate-certificate}
  \State \textbf{if} $\lambda\Delta_{\rho/\lambda}(\bar{y}^{s};a)
  <(\lambda-1)\beta\Delta_s$, \textbf{then} $\underline f\gets\ell_s$
  \Comment{{\scriptsize RBCer certifies lower bound}} 
  \label{line:lower_bound_improve}
  \State \textbf{if} $f(\hat{x})-\underline f\leq\beta\Delta_s$, \textbf{then}
  $\bar{y}^{s+1}\gets\hat{x}$; \textbf{break}
  \Comment{{\scriptsize Update upper bound}}
  \label{line:test-condition}
 \EndWhile
\EndFor
\end{algorithmic}
\end{algorithm}

Each APBM-One-Step call uses at most $B+1$ new oracle queries and stores
at most $B+3$ affine functions in its model: the incoming aggregate,
the level $\ell$, the initial cut at $\hat{x}^{t}$, and $B$ new cuts.
The oracle information at the returned incumbent is retained for the next call.
Lemma~\ref{lem:alm-model-growth} shows that
dual aggregation preserves the final proximal minimizer and minimum.
Hence, line~\ref{line:alm-evaluate-certificate} of Alg.~\ref{alg:accelerated-lm-pbm} computes the regularized bundle
certificate at the returned point $x=x^{t,B}$ without an additional
subproblem or oracle query.

Alg.~\ref{alg:accelerated-lm-pbm} retains the lower bound $\underline f$
and raises it to $\ell_s$ in line~\ref{line:lower_bound_improve}.
Line~\ref{line:test-condition} of Alg.~\ref{alg:accelerated-lm-pbm}
restarts once $f(\hat{x})-\underline f\leq\beta\Delta_s$.
If the lower bound is unchanged, this is equivalent to
$f(\hat{x})\leq f(\bar{y}^{s})-(1-\beta)\Delta_s$; after a lower-bound update,
it holds because $f(\hat{x})\leq f(\bar{y}^{s})$ and
$f(\bar{y}^{s})-\ell_s=\beta\Delta_s$.
The restart takes $\hat{x}$ as the next reference and resets $(\lambda,t)$ to
$(0,0)$.  The aggregate $a$ need not be reinitialized, since its initial
weight is zero. Let $\underline f_s$ denote the lower bound at the start
of stage $s$.
Completing the square in~\eqref{eq:rbcer} gives
$\Delta_{\rho/2}(\bar{y}^{1};\ell_f(\cdot;\bar{y}^{1}))
\leq\norm{f^{\prime}(\bar{y}^{1})}^2/\rho=\Delta_1/2$.
For $0<\rho\leq\mu^*$, this bound and
Propositions~\ref{prop:rbcer-objective-gap} and~\ref{prop:alm-model-test}
guarantee $\underline f_s\leq f^*$ throughout the algorithm.
Each restart satisfies $\Delta_{s+1}\leq\beta\Delta_s$, yielding
$f(\bar{y}^{s})-f^*\leq\Delta_s\leq\beta^{s-1}\Delta_1$.
The following results combine this contraction of the computable gap
$\Delta_s$ with bounds on the oracle work per stage to establish
convergence results.

\textit{Comparison with classical PBMs.}
In Alg.~\ref{alg:lm-pbm}, the oracle is queried at the proximal minimizer,
and the serious-step test determines whether it becomes the next center.
APBMs retain affine cuts, quadratic regularization, and aggregation,
but query the Nesterov interpolation in line~\ref{line:alm-interpolate}
of Alg.~\ref{alg:alm-block}, while keeping the reference $\bar{y}^{s}$ fixed
within each stage even as the incumbent improves.
The classical serious/null-step test is replaced by the gap reduction test
in line~\ref{line:test-condition} of Alg.~\ref{alg:accelerated-lm-pbm}:
sufficient objective decrease or an improvement of the lower bound
certified by the $\RBCer$.

\textbf{Convergence results.} 
Proposition~\ref{prop:alm-model-test} establishes a conditional lower bound on $f^*$.
Theorems~\ref{thm:alm-complexity} and~\ref{thm:alm-universal} establish
oracle complexity under PWS and H\"older smoothness, respectively.

% Proposition~\ref{prop:alm-model-test} bounds the objective gap whenever
% the certificate test in line~\ref{line:test-condition} of Alg.~\ref{alg:accelerated-lm-pbm} holds.
\begin{proposition}\label{prop:alm-model-test}
Let $\bar{y}\in X$, $\ell<f(\bar{y})$,
$\rho>0$, and $\lambda>0$.
Suppose an affine minorant $a\leq f$ satisfies
$\Delta_{\rho/\lambda}(\bar{y};a)
 \leq(1-1/\lambda)(f(\bar{y})-\ell)$.
If $f$ is $\mu^*$-QG, then
$f(\bar{y})-f^*\leq\max\{1,\rho/\mu^*\}(f(\bar{y})-\ell)$.
\end{proposition}

Building on this certificate bound, Theorems~\ref{thm:alm-complexity}
and~\ref{thm:alm-universal} establish the oracle complexity of
Alg.~\ref{alg:accelerated-lm-pbm} under PWS and H\"older smoothness,
respectively.  The algorithm requires neither the smoothness constants
nor the H\"older exponent.
\begin{thm}
\label{thm:alm-complexity}
Suppose $f$ satisfies the QG condition and Assumption~\ref{ass:first-order-oracle}
holds.  Let $B\geq k$, $\rho=\mu^*$, and $\beta\in(1/2,1)$.
For any $\vep>0$, Alg.~\ref{alg:accelerated-lm-pbm} finds an
$\vep$-optimal point $\bar{y}^{s}$ satisfying
$f(\bar{y}^{s})-f^*\leq\Delta_s\leq\vep$ within
{\small
\begin{equation}\label{eq:alm-complexity}
 O\!\left(
 B\left(1+\sqrt{\frac{L}{\mu^*}}\right)
 \left[1+\log_+\frac{2\norm{f^{\prime}(\bar{y}^{1})}^2}{\mu^*\vep}\right]
 \right)
\end{equation}
}
first-order oracle calls.
\end{thm}
This bound replaces the factor $1+L/\mu^*$ in
Theorem~\ref{thm:lm-rbcer-complexity} by $1+\sqrt{L/\mu^*}$, matching the best known complexity result in~\cite{partOne}.

Replacing the PWS assumption by the H\"older condition~\eqref{eq:alm-holder-model},
Theorem~\ref{thm:alm-universal} gives the oracle complexity of the same
algorithm for any $B\geq1$.
\begin{thm}
\label{thm:alm-universal}
Suppose $f$ satisfies the QG condition with $0<\mu^*<\infty$
and~\eqref{eq:alm-holder-model} holds.
Let $B\geq1$, $\rho=\mu^*$, and $\beta\in(1/2,1)$.
For any $\vep>0$, Alg.~\ref{alg:accelerated-lm-pbm} finds an
$\vep$-optimal point $\bar{y}^{s}$ satisfying
$f(\bar{y}^{s})-f^*\leq\Delta_s\leq\vep$ within
{\footnotesize
\begin{equation}\label{eq:alm-universal-complexity}
 \begin{cases}
 O\!\left(B\left[1+\log_+\frac{\Delta_1}{\vep}
 +\frac{M_0^2}{\mu^*\vep}\right]\right),
 &\nu=0,\\[3pt]
 O\!\left(B\left[1+\log_+\frac{\Delta_1}{\vep}
 +\left(\frac{M_\nu^2}{(\mu^*)^{1+\nu}\vep^{1-\nu}}
 \right)^{\frac{1}{1+3\nu}}
 \min\left\{1+\log_+\frac{\Delta_1}{\vep},
 \frac{1}{1-\beta^{r(\nu)}}\right\}\right]\right),
 &0<\nu<1,\\[3pt]
 O\!\left(B\left(1+\sqrt{\frac{M_1}{\mu^*}}\right)
 \left[1+\log_+\frac{\Delta_1}{\vep}\right]\right),
 &\nu=1.
 \end{cases}
\end{equation}
}
first-order oracle calls, where
$r(\nu)=(1-\nu)/(1+3\nu)$ and
$\Delta_1=2\norm{f^{\prime}(\bar{y}^{1})}^2/\mu^*$.
\end{thm}

% For $B=1$, these bounds
% match the leading dependence on smoothness, growth, and accuracy in
% the lower bounds of~\cite{doikov2025lower,nesterov2018lectures} for
% each fixed $\nu$. 

\paragraph{Comparison with lower bounds.}
For $B=1$ and $\rho=\mu^*$, the leading joint dependence on
$M_\nu$, $\mu^*$, and $\vep$ is optimal for fixed $\nu$ in the
deterministic first-order oracle model with unrestricted dimension.
For $0\leq\nu<1$, use Doikov's ball-constrained problem
\cite[(1), Theorem~5 and Remark~1]{doikov2025lower} with quadratic regularization.
To match the whole-objective condition~\eqref{eq:alm-holder-model},
take its hard convex part $q$ with H\"older constant $M_\nu/2$ and consider
$F(x,z)=q(x)+(\mu^*/2)(\norm{x}^2+z^2)$ on
$\{(x,z):\norm{x}\leq D,\ |z|\leq D\}$, where
$\mu^*(2\sqrt2D)^{1-\nu}=M_\nu/(2(1+\nu))$.
For two points at distance $d$ in this set, the linearization error is at most
$\frac{M_\nu}{2(1+\nu)}d^{1+\nu}+\frac{\mu^*}{2}d^2
\leq\frac{3M_\nu}{4(1+\nu)}d^{1+\nu}$.
Strong convexity ensures the QG condition with modulus $\mu^*$. Varying only $z$
at a minimizer proves that this is exactly the maximal modulus in the QG condition.
One query to $q$ simulates one query to $F$, since the quadratic is known.
Above the cited theorem's query threshold, which this radius choice makes
dependent only on $\nu$, its lower bound yields the leading term
in~\eqref{eq:alm-universal-complexity}.
At $\nu=0$, the construction can have $\norm{\nabla q}\leq M_0/4$.
The quadratic gradient is also at most $M_0/4$ on this set, so
$\norm{\nabla F}\leq M_0/2$.
At $\nu=1$, the classical smooth strongly convex lower bound
\cite[Section~2.1]{nesterov2018lectures} gives the square-root
condition-number dependence and accuracy logarithm.
To transfer output-point lower bounds to the
best queried point, append one query at that point; the extra call does
not change the leading order.
Initialization terms remain explicit. No optimality in $B,k$ or
uniform endpoint constants is claimed.

The method adapts to unknown smoothness in the sense
of~\cite{lan2015bundle}; these guarantees use the QG modulus to choose $\rho$.

\subsection{Parameter-free accelerated proximal bundle methods}
\label{subsec:alm-adaptive}
The convergence rates in Theorems~\ref{thm:alm-complexity}
and~\ref{thm:alm-universal} are attained with $\rho=\mu^*$.
When $\mu^*$ is unknown, this choice is unavailable, and taking
$\rho>\mu^*$ may invalidate the lower bounds in
Alg.~\ref{alg:accelerated-lm-pbm}.
We therefore develop Alg.~\ref{alg:adaptive-apbm}, which starts from an
upper estimate of $\mu^*$ and uses the $\RBCer$ to adjust $\rho$ whenever
a queried objective value violates the conditional lower bound.

In outer iteration $s$, the adaptive method attempts to reduce the
certificate scale from $\Delta_s$ to $\Delta=\beta\Delta_s$.
Its stronger test
$\lambda\Delta_{\rho/\lambda}(\bar{y};a)\leq(\lambda-2)\Delta$
allows Proposition~\ref{prop:alm-model-test} to be applied with
parameters $\rho/2$ and $\lambda/2$.
Under QG, it gives
$f(\bar{y})-f^*\leq\max\{1,\rho/(2\mu^*)\}\Delta$.
Line~\ref{line:apbm-lower-bound} of Alg.~\ref{alg:adaptive-apbm}
therefore takes the maximum of the recorded lower bound and
$f(\bar{y}^{s})-\max\{\Delta_s,\Delta/(2\beta)\}$.
The record is retained across outer iterations while $\rho$ is unchanged
and reset to $-\infty$ after each halving.  Thus, $\underline f$ is
nondecreasing between halvings.
Lemma~\ref{lem:alm-halving} shows that every bound in this maximum is
valid whenever the current $\rho\leq\mu^*$; hence $f(\hat{x})<\underline f$
implies $\rho>\mu^*$ and justifies halving.

Alg.~\ref{alg:adaptive-apbm} first checks whether $f(\hat{x})<\underline f$.
If so, line~\ref{line:apbm-update} halves $\rho$, doubles $\Delta$,
and clears the recorded lower bound, retaining $(\bar{y},a,\lambda)$.
Otherwise, the outer iteration ends if
$\lambda\Delta_{\rho/\lambda}(\bar{y};a)\leq(\lambda-2)\Delta$.
If this inequality fails and $f(\hat{x})<f(\bar{y})-\Delta/(2\beta)$, line~\ref{line:apbm-reference}
sets $(\bar{y},\lambda)\gets(\hat{x},0)$.
A parameter or reference update restarts APBM-One-Step at $t=1$.
In the halving branch, $a\leq f$ remains valid even though the conditional
lower bound is disproved.  Keeping this weighted minorant at the same reference
preserves certificate progress across halvings
(Lemma~\ref{lem:alm-parameter-transfer}); jointly bounding the work across
these updates then gives additive search overhead
(Lemma~\ref{lem:alm-search-progress}).

\begin{algorithm}[!htbp]
\caption{Parameter-free APBMs}
\label{alg:adaptive-apbm}
\begin{algorithmic}[1]
\Require $\bar{y}^{1}\in X$, $\widehat\rho_1>0$, $B\in\mathbb N_+$, $\beta\in(1/2,1)$
\State \textbf{If} $\norm{f^{\prime}\brbra{\bar{y}^{1}}}=0$, \textbf{then} \textbf{return} $\bar{y}^{1}$ 
\State $a_1\gets\ell_f(\cdot;\bar{y}^{1})$, $\underline f\gets-\infty$,
$\Delta_1\gets4\norm{f^{\prime}(\bar{y}^{1})}^2/\widehat\rho_1$;
\label{line:apbm-initialize}
\For{$s=1,2,\ldots$}
 \State $(\bar{y},\rho,\Delta,a,\lambda)\gets
 (\bar{y}^{s},\widehat\rho_s,\beta\Delta_s,a_s,0)$;
 $(\hat{x},t)\gets(\bar{y},1)$
 \label{line:apbm-stage}
 \While{\textbf{true}}
  \label{line:apbm-repeat}
  \State $\underline f\gets\max\{\underline f,\,f(\bar{y}^{s})-\max\{\Delta_s,\frac{\Delta}{2\beta}\}\}$;
  $\ell\gets f(\bar{y})-\Delta$
  \label{line:apbm-lower-bound}
  \State $(\hat{x},x,\lambda,a)\gets\text{APBM-One-Step}(\bar{y},\hat{x},\ell,\rho,B,t,\lambda,a)$
  \label{line:apbm-step}
  \State $\displaystyle\Delta_{\rho/\lambda}(\bar{y};a)\gets
  f(\bar{y})-a(x)-\frac{\rho}{2\lambda}\norm{x-\bar{y}}^2$
  \label{line:apbm-certificate}
  \State \textbf{if} $f(\hat{x})<\underline f$, \textbf{then}
  $(\rho,\Delta,\underline f,\hat{x},t)\gets(\rho/2,2\Delta,-\infty,\bar{y},1)$
  \Comment{{\scriptsize Find $\rho > \mu^*$}}
  \label{line:apbm-update}\label{line:apbm-restart}
  \State \textbf{else if} $\lambda\Delta_{\rho/\lambda}(\bar{y};a)
  \leq(\lambda-2)\Delta$, \textbf{then break}
  \Comment{{\scriptsize RBCer certifies lower bound}}
  \label{line:apbm-certify}
  \State \textbf{else if} $f(\hat{x})<f(\bar{y})-\tfrac{\Delta}{2\beta}$,
  \textbf{then} $(\bar{y},\lambda,t)\gets(\hat{x},0,1)$
  \Comment{{\scriptsize Improve upper bound}}
  \label{line:apbm-trial-test}\label{line:apbm-reference}
  \State \textbf{else} $t\gets t+1$
 \EndWhile
 \State $(\bar{y}^{s+1},\widehat\rho_{s+1},\Delta_{s+1},a_{s+1})\gets(\bar{y},\rho,\Delta,a)$
\EndFor
\end{algorithmic}
\end{algorithm}

\textbf{Convergence results.} 
The following bounds retain an additive logarithmic cost for parameter search.

\begin{thm}\label{thm:alm-adaptive}
Suppose $f$ satisfies the QG condition, $\widehat\rho_1\geq\mu^*>0$,
and $\beta\in(1/2,1)$.  For any $\vep>0$,
Alg.~\ref{alg:adaptive-apbm} produces a $\bar{y}^{s}$ satisfying
$f(\bar{y}^{s})-f^*\leq\vep$ within
{\small
\begin{equation}\label{eq:alm-adaptive-pws}
 O\!\left(B\left[
 \left(1+\sqrt{\frac L{\mu^*}}\right)
 \left(1+\log_+\frac{\norm{f^{\prime}(\bar{y}^{1})}^2}{\mu^*\vep}\right)
 +\log_+\frac{\widehat\rho_1}{\mu^*}\right]\right)
\end{equation}
}
first-order oracle calls under Assumption~\ref{ass:first-order-oracle}
with $B\geq k$.  Under~\eqref{eq:alm-holder-model} with $B\geq1$ and
$0\leq\nu<1$, Alg.~\ref{alg:adaptive-apbm} finds an $\vep$-optimal
solution within
{\scriptsize
\begin{equation}\label{eq:alm-adaptive-universal}
O\!\Biggl(B\Biggl[
 1+\log_+\frac{\norm{f^{\prime}(\bar{y}^{1})}^2}{\mu^*\vep}
 +\log_+\frac{\widehat\rho_1}{\mu^*}
 +\left(\frac{M_\nu^2}{(\mu^*)^{1+\nu}\vep^{1-\nu}}
 \right)^{\frac1{1+3\nu}}
 \min\left\{1+\log_+\frac{\norm{f^{\prime}(\bar{y}^{1})}^2}{\mu^*\vep},
 \frac{1}{1-\beta^{r(\nu)}}\right\}\Biggr]\Biggr)
\end{equation}
}
first-order oracle calls, where $r(\nu)=(1-\nu)/(1+3\nu)$.
For $\nu=1$, \eqref{eq:alm-adaptive-pws} applies with $L=M_1$ and $B\geq1$.
If $\nu=0$ and $\norm{f^{\prime}(x)}\leq M_0/2$ on $X$, the oracle complexity is
$O(B[1+M_0^2/(\mu^*\vep)+\log_+(\widehat\rho_1/\mu^*)])$.
\end{thm}

To ensure $\widehat\rho_1\geq\mu^*$, given an already queried
$u\in X$ with $f(u)<f(\bar{y}^{1})$, choose
$\widehat\rho_1=2\norm{f^{\prime}(\bar{y}^{1})}^2/(f(\bar{y}^{1})-f(u))$.
Under the QG condition, completing the square in~\eqref{eq:rbcer}
and applying Proposition~\ref{prop:rbcer-objective-gap} with proximal
parameter $\mu^*/2$ give
$f(\bar{y}^{1})-f(u)\leq2\norm{f^{\prime}(\bar{y}^{1})}^2/\mu^*$, implying
$\widehat\rho_1\geq\mu^*$.
With this initialization, $B\geq k$ gives the best-known PWS dependence
on $L/\mu^*$ and $\vep$~\cite{partOne}, while $B=1$ gives an anytime
algorithm with optimal leading complexity across the global H\"older classes.

\subsection{Proof details}\label{subsec:alm-proofs}
We first prove Theorems~\ref{thm:alm-complexity} and~\ref{thm:alm-universal}
from Subsection~\ref{subsec:alm-algorithm}, then Theorem~\ref{thm:alm-adaptive}
from Subsection~\ref{subsec:alm-adaptive}.
Specifically, the proof proceeds in two parts.

\emph{Acceleration.}
Proposition~\ref{prop:alm-model-test} justifies lower-bound updates, and
Lemma~\ref{lem:alm-model-growth} shows that dual aggregation leaves the
regularized subproblem minimum unchanged.  Proposition~\ref{prop:alm-block-progress}
bounds the objective residual under PWS or H\"older smoothness;
Lemma~\ref{lem:alm-stage} converts these bounds into stage costs, yielding
Theorems~\ref{thm:alm-complexity} and~\ref{thm:alm-universal}.

\emph{Parameter-free acceleration.}
Lemmas~\ref{lem:alm-parameter-transfer} and~\ref{lem:alm-halving} justify
the parameter updates and bound their number.
Lemma~\ref{lem:alm-search-progress} bounds search costs at a fixed reference.
Together with the single possible
reference change, these bounds give Theorem~\ref{thm:alm-adaptive}.

For fixed $\bar{y}$, $\ell$, and $\rho$, let $a^t$ and $\lambda_t$
denote the aggregate and its weight after $t$ APBM-One-Step calls, with
$a^0$ and $\lambda_0$ denoting their initial values.  Write
{\small
\begin{align}
 \Phi_t&:=\min_{x\in X}\left\{\lambda_t(a^t(x)-\ell)
                   +\frac\rho2\norm{x-\bar{y}}^2\right\},
 \label{eq:alm-model-minimum}\\
 \Phi_{t,i}&:=\min_{x\in X}\left\{\lambda_{t-1}(a^{t-1}(x)-\ell)
       +t(\psi^{t,i}(x)-\ell)+\frac\rho2\norm{x-\bar{y}}^2\right\},\notag\\
 \Phi_t&=\lambda_t\left[f(\bar{y})-\ell
                  -\Delta_{\rho/\lambda_t}(\bar{y};a^t)\right],
 \qquad \lambda_t>0.
 \label{eq:alm-certificate-identity}
\end{align}
}
Identity~\eqref{eq:alm-certificate-identity} follows from~\eqref{eq:rbcer};
when $\lambda_t=0$, the defining minimum gives $\Phi_t=0$.

\begin{proof}[\underline{Proof of Proposition~\ref{prop:alm-model-test}}]
Let $\Phi$ denote the minimum in~\eqref{eq:alm-model-minimum}
with aggregate $a$ and weight $\lambda$.
By~\eqref{eq:alm-certificate-identity} and $\lambda>0$, the assumption
$\Delta_{\rho/\lambda}(\bar{y};a)\leq(1-1/\lambda)(f(\bar{y})-\ell)$
implies $\Phi\geq f(\bar{y})-\ell$.
The conclusion is immediate when $\ell\leq f^*$.
We now consider the case $\ell>f^*$.  Choose a nearest optimizer $x^*$ and set
$w=\bar{y}+\frac{f(\bar{y})-\ell}{f(\bar{y})-f^*}(x^*-\bar{y})$.
Since $f^*<\ell<f(\bar{y})$, we have
$0<\frac{f(\bar{y})-\ell}{f(\bar{y})-f^*}<1$.
Thus $w\in X$ by convexity of $X$.  Using $a\leq f$, convexity of $f$,
and $f(x^*)=f^*$, we obtain
{\small
\begin{equation*}
 a(w)\leq f(w)
 \leq\frac{\ell-f^*}{f(\bar{y})-f^*}f(\bar{y})
     +\frac{f(\bar{y})-\ell}{f(\bar{y})-f^*}f^*
 =\ell.
\end{equation*}
}
Evaluating the minimum defining $\Phi$ at $w$ and using
$\Phi\geq f(\bar{y})-\ell$ and $a(w)\leq\ell$ gives
{\small
\begin{align*}
 f(\bar{y})-\ell
 \leq\Phi\leq\lambda(a(w)-\ell)+\frac\rho2\norm{w-\bar{y}}^2
 \leq\frac\rho2\norm{w-\bar{y}}^2
 =\frac{\rho(f(\bar{y})-\ell)^2}{2(f(\bar{y})-f^*)^2}
   \norm{x^*-\bar{y}}^2.
\end{align*}
}
Under QG, substituting
$\norm{x^*-\bar{y}}^2\leq2(f(\bar{y})-f^*)/\mu^*$ and dividing by
$f(\bar{y})-\ell>0$ gives
$f(\bar{y})-f^*\leq(\rho/\mu^*)(f(\bar{y})-\ell)$.
\end{proof}

The dual aggregate preserves this minimum and permits certificate evaluation.
\begin{lem}\label{lem:alm-model-growth}
Suppose APBM-One-Step (Alg.~\ref{alg:alm-block}) starts from an affine minorant
$a^0\leq f$ and $\lambda_0\geq0$, with $\bar{y}\in X$,
level $\ell\in\mathbb R$, parameters $\rho>0$ and
$B\in\mathbb N_+$.
For every $t\geq1$, the resulting aggregate satisfies $a^t(x)\leq f(x)$, and
{\small
\begin{gather}
 \Phi_{t-1}\leq\Phi_{t,0}\leq\cdots\leq\Phi_{t,B}=\Phi_t,
 \label{eq:alm-monotonicity}\\
 \Phi_{t,j}-\Phi_{t,i}\geq\frac\rho2\norm{x^{t,j}-x^{t,i}}^2,
 \qquad 0\leq i<j\leq B.
 \label{eq:alm-proximal-growth}
\end{gather}
}
If, in addition, $\hat{x}^{1}=\bar{y}$, $\lambda_0=0$, and $\ell<f(\bar{y})$, then
$\lambda_t>0$ for all $t\geq1$.
\end{lem}

\begin{proof}
We first verify the minorant property.  By~\eqref{eq:alm-dual-aggregate},
$a^t$ is a convex combination of $a^{t-1}$ and oracle cuts when
$\lambda_t>0$, and an oracle cut when $\lambda_t=0$.
Induction from $a^0\leq f$ therefore proves $a^t(x)\leq f(x)$ for all $x\in X$.

To prove~\eqref{eq:alm-monotonicity}, note that
$\psi^{t,0}\geq\ell$, so the objective defining $\Phi_{t,0}$
is pointwise no smaller than the one defining $\Phi_{t-1}$.
Taking minima gives $\Phi_{t,0}\geq\Phi_{t-1}$.
For the final equality, consider the last subproblem in
line~\ref{line:alm-proximal} of Alg.~\ref{alg:alm-block}.
Strict feasibility of its affine epigraph constraints ensures the
existence of optimal dual multipliers.  The aggregate update gives
$\lambda_t(a^t(x)-\ell)=\lambda_{t-1}(a^{t-1}(x)-\ell)
+\sum_{j=0}^B\gamma_j(c_j(x)-\ell)$ for all $x\in X$.
Dual feasibility bounds this affine term above by the uncompressed model,
with equality at $x^{t,B}$ by complementary slackness.
Stationarity gives
$n^t:=\rho(\bar{y}-x^{t,B})-\lambda_t\nabla a^t(x^{t,B})\in N_X(x^{t,B})$,
so $x^{t,B}$ also minimizes the objective defining $\Phi_t$.
The two objectives agree at $x^{t,B}$, proving $\Phi_{t,B}=\Phi_t$.
This argument also applies when $\lambda_t=0$, since the retained affine
term then vanishes.

For the inequalities, line~\ref{line:alm-add-cut} of Alg.~\ref{alg:alm-block} gives
$\psi^{t,j}(x)\geq\psi^{t,i}(x)$ for all $x\in X$ whenever $i<j$.
Evaluating the objective of subproblem $i$ at $x^{t,j}$ therefore gives
a value no larger than $\Phi_{t,j}$.
Since this objective is $\rho$-strongly convex and minimized at $x^{t,i}$,
the same value is at least
$\Phi_{t,i}+\rho\norm{x^{t,j}-x^{t,i}}^2/2$,
proving~\eqref{eq:alm-proximal-growth}.
Its nonnegative right-hand side, together with
$\Phi_{t,0}\geq\Phi_{t-1}$ and $\Phi_{t,B}=\Phi_t$,
completes the proof of~\eqref{eq:alm-monotonicity}.

It remains to prove $\lambda_t>0$ under the additional assumptions.
Since $\hat{x}^{1}=\bar{y}$, the initial bundle contains
$c_0=\ell_f(\cdot;\bar{y})$, which is retained throughout the call.
Thus $\psi^{1,B}(\bar{y})\geq c_0(\bar{y})=f(\bar{y})>\ell$.
If $\lambda_1=0$, \eqref{eq:alm-dual-aggregate} and $\lambda_0=0$ force
$\gamma_j=0$ for $j=0,\ldots,B$, so $\gamma_\ell=1$.
Stationarity then gives $x^{1,B}=\bar{y}$.
Complementary slackness for the constant plane, whose multiplier is
$\gamma_\ell>0$, gives $\psi^{1,B}(\bar{y})=\ell$, a contradiction.
Thus $\lambda_1>0$, and~\eqref{eq:alm-dual-aggregate} yields
$\lambda_t\geq\lambda_{t-1}>0$ for every $t\geq2$, proving the final assertion.
\end{proof}

The lemma also holds if the initial cut $c_0$ is omitted from the bundle
and aggregation, so that $\psi^{t,0}\equiv\ell$ and $\gamma_0=0$.
In this variant, $\Phi_{t,0}=\Phi_{t-1}$.
When $\lambda_0=0$ and $\hat{x}^{1}=\bar{y}$, the first subproblem and
$\alpha_1=1$ give $x^{1,0}=\bar{x}^{1,0}=\bar{y}$.
Since $B\geq1$, its oracle cut enters $\psi^{1,B}$, and the same
argument proves $\lambda_1>0$.
The subsequent progress estimates and complexity bounds apply to both
initializations: they use only~\eqref{eq:alm-monotonicity},
\eqref{eq:alm-proximal-growth}, and the cuts added during each call.

The next proposition bounds the objective residual after successive APBM-One-Step calls.

\begin{proposition}
\label{prop:alm-block-progress}
Suppose APBM-One-Step (Alg.~\ref{alg:alm-block}) is run with fixed
$\bar{y}\in X$, $\ell\in\mathbb R$, and $\rho>0$, starting from
$\hat{x}^{1}=\bar{y}$, an affine minorant $a^0\leq f$, and $\lambda_0\geq0$.
Let $\omega_t=t(t+1)/2$ for $t\geq0$.
Under Assumption~\ref{ass:first-order-oracle} and $B\geq k$, for every $t\geq1$,
{\small
\begin{equation}\label{eq:alm-accumulated-progress}
 \omega_t\bigl(f(\hat{x}^{t+1})-\ell\bigr)
 \leq\left(1+\frac{2L}{\rho}\right)(\Phi_t-\Phi_0).
\end{equation}
}
Under~\eqref{eq:alm-holder-model} and $B\geq1$, the corresponding bound is
{\small
\begin{equation}\label{eq:alm-warm-progress}
 \omega_t\bigl(f(\hat{x}^{t+1})-\ell\bigr)\leq\Phi_t-\Phi_0
 +2M_\nu\left(\frac{\Phi_t-\Phi_0}{\rho}\right)^{\frac{1+\nu}{2}}
 t^{\frac{3(1-\nu)}2},\qquad 0\leq\nu\leq1.
\end{equation}
}
\end{proposition}

\begin{proof}
The definition of $\omega_t$ and the choice of $\alpha_t$ in
line~\ref{line:alm-initialize} of Alg.~\ref{alg:alm-block} give
{\small
\begin{equation}\label{eq:alm-weight-identities}
 \omega_t(1-\alpha_t)=\omega_{t-1},\qquad
 t=\omega_t\alpha_t=\omega_t-\omega_{t-1},\qquad
 \omega_t\alpha_t^2={2t}/\brbra{t+1}<2.
\end{equation}
}
In the following, we consider the two oracle assumptions separately.

\emph{Piecewise-smooth.}
Among the $B+1\geq k+1$ queries, choose $i<j$ whose interpolated
points share a piece.  The earlier cut is retained in $\psi^{t,j}$,
and line~\ref{line:alm-best-point} of Alg.~\ref{alg:alm-block} changes only $\hat{x}^{t+1}$, keeping
$\hat{x}^{t}$ fixed throughout this call.  Thus
$\bar{x}^{t,j}-\bar{x}^{t,i}=\alpha_t(x^{t,j}-x^{t,i})$, and
the same-piece oracle bound gives
{\small
\begin{equation}\label{eq:alm-same-piece}
\begin{aligned}
 f(\bar{x}^{t,j})
 &\leq\ell_f(\bar{x}^{t,j};\bar{x}^{t,i})
 +\frac{L\alpha_t^2}{2}\norm{x^{t,j}-x^{t,i}}^2\\
 &=(1-\alpha_t)\ell_f(\hat{x}^{t};\bar{x}^{t,i})
   +\alpha_t\ell_f(x^{t,j};\bar{x}^{t,i})
   +\frac{L\alpha_t^2}{2}\norm{x^{t,j}-x^{t,i}}^2\\
 &\leq(1-\alpha_t)f(\hat{x}^{t})+\alpha_t\psi^{t,j}(x^{t,j})
 +\frac{L\alpha_t^2}{2}\norm{x^{t,j}-x^{t,i}}^2.
\end{aligned}
\end{equation}
}
Since $x^{t,j}$ attains $\Phi_{t,j}$ in line~\ref{line:alm-proximal} of Alg.~\ref{alg:alm-block},
the definition of $\Phi_{t-1}$ in~\eqref{eq:alm-model-minimum} gives
{\small
\begin{equation}\label{eq:alm-weight-growth}
\begin{aligned}
 &(\omega_t-\omega_{t-1})
 \bigl(\psi^{t,j}(x^{t,j})-\ell\bigr)
 =t\bigl(\psi^{t,j}(x^{t,j})-\ell\bigr)\\
 &=\Phi_{t,j}-\lambda_{t-1}\bigl(a^{t-1}(x^{t,j})-\ell\bigr)
   -\frac{\rho}{2}\norm{x^{t,j}-\bar{y}}^2\leq\Phi_{t,j}-\Phi_{t-1}.
\end{aligned}
\end{equation}
}
Line~\ref{line:alm-best-point} of Alg.~\ref{alg:alm-block} gives
$f(\hat{x}^{t+1})\leq f(\bar{x}^{t,j})$.
Subtracting $\ell$ from~\eqref{eq:alm-same-piece}, multiplying by $\omega_t$,
and applying~\eqref{eq:alm-weight-identities}, \eqref{eq:alm-weight-growth},
and~\eqref{eq:alm-proximal-growth}, we obtain
{\small
\begin{equation*}
\begin{aligned}
 \omega_t\bigl(f(\hat{x}^{t+1})-\ell\bigr)
  -\omega_{t-1}\bigl(f(\hat{x}^{t})-\ell\bigr)&\leq (\Phi_{t,j}-\Phi_{t-1})
  +\frac{L\omega_t\alpha_t^2}{\rho}(\Phi_{t,j}-\Phi_{t,i})\\
 &\leq\left(1+\frac{2L}{\rho}\right)(\Phi_t-\Phi_{t-1}).
\end{aligned}
\end{equation*}
}
The last inequality uses $\omega_t\alpha_t^2<2$ and
monotonicity~\eqref{eq:alm-monotonicity}.
Summing over the calls and using $\omega_0=0$ proves
\eqref{eq:alm-accumulated-progress} by telescoping.

\emph{H\"older smoothness.}
Since $B\geq1$, line~\ref{line:alm-add-cut} of Alg.~\ref{alg:alm-block} gives
$\psi^{t,1}(x)\geq\ell_f(x;\bar{x}^{t,0})$ for every $x\in X$.
Using~\eqref{eq:alm-holder-model} with $x=\bar{x}^{t,0}$ and
$y=\bar{x}^{t,1}$ in the derivation of~\eqref{eq:alm-same-piece},
and applying line~\ref{line:alm-best-point} of Alg.~\ref{alg:alm-block}, we obtain
{\small
\begin{equation}\label{eq:alm-universal-progress}
\begin{aligned}
 f(\hat{x}^{t+1})-(1-\alpha_t)f(\hat{x}^{t})-\alpha_t\ell
 &\leq\alpha_t\bigl(\psi^{t,1}(x^{t,1})-\ell\bigr)
 +\frac{M_\nu\alpha_t^{1+\nu}}{1+\nu}
  \norm{x^{t,1}-x^{t,0}}^{1+\nu}\\
 &\leq\frac{\Phi_t-\Phi_{t-1}}{\omega_t}
 +\frac{M_\nu}{1+\nu}
\left(\frac{2\alpha_t^2(\Phi_t-\Phi_{t-1})}{\rho}\right)^{\frac{1+\nu}{2}}.
\end{aligned}
\end{equation}
}
The last inequality follows from~\eqref{eq:alm-weight-identities},
\eqref{eq:alm-weight-growth}, and Lemma~\ref{lem:alm-model-growth}.
Multiplying~\eqref{eq:alm-universal-progress} by $\omega_t$ and summing
over the calls telescopes the objective terms by~\eqref{eq:alm-weight-identities},
with $\omega_0=0$.  It remains to bound the sum of the H\"older error terms.
For $0\leq\nu<1$, we have
{\small
\begin{align*}
 \sum_{j=1}^t\omega_j\alpha_j^{1+\nu}(\Phi_j-\Phi_{j-1})^{\frac{1+\nu}{2}}
 \leq2^\nu\left(\sum_{j=1}^t j^2\right)^{\frac{1-\nu}{2}}
 (\Phi_t-\Phi_0)^{\frac{1+\nu}{2}}\leq2^\nu t^{\frac{3(1-\nu)}2}(\Phi_t-\Phi_0)^{\frac{1+\nu}{2}},
\end{align*}
}
where the first inequality uses the coefficient bound
$\omega_j\alpha_j^{1+\nu}=2^\nu j/(j+1)^\nu\leq2^\nu j^{1-\nu}$
and H\"older's inequality with conjugate exponents $2/(1-\nu)$ and $2/(1+\nu)$:
{\small
\begin{equation*}
 \sum_{j=1}^t (j^2)^{\frac{1-\nu}{2}}
 (\Phi_j-\Phi_{j-1})^{\frac{1+\nu}{2}}
 \leq\left(\sum_{j=1}^t j^2\right)^{\frac{1-\nu}{2}}
 \left(\sum_{j=1}^t(\Phi_j-\Phi_{j-1})\right)^{\frac{1+\nu}{2}}.
\end{equation*}
}
Here~\eqref{eq:alm-monotonicity} ensures $\Phi_j-\Phi_{j-1}\geq0$,
and $\sum_{j=1}^t(\Phi_j-\Phi_{j-1})=\Phi_t-\Phi_0$.
The last inequality uses $\sum_{j=1}^t j^2\leq t^3$.
Substituting this estimate into the summed bound and using
$2^{(1+3\nu)/2}/(1+\nu)\leq2$ proves~\eqref{eq:alm-warm-progress}.
For $\nu=1$, \eqref{eq:alm-weight-identities} instead gives
$\sum_{j=1}^t\omega_j\alpha_j^2(\Phi_j-\Phi_{j-1})
\leq2(\Phi_t-\Phi_0)$, which yields the same conclusion.
\end{proof}

The next lemma bounds the number of APBM-One-Step calls per stage.
\begin{lem}\label{lem:alm-stage}
Consider the APBM-One-Step sequence in Proposition~\ref{prop:alm-block-progress},
with $\Phi_0\geq0$ and $\ell=f(\bar{y})-\beta\Delta$,
where $\Delta>0$ and $\beta\in(1/2,1)$.
A stage ends as soon as $f(\hat{x}^{t+1})\leq f(\bar{y})-(1-\beta)\Delta$
or $\Phi_t>\beta\Delta$.
Under Assumption~\ref{ass:first-order-oracle} and $B\geq k$,
the stage terminates after at most
{\small
\begin{equation}\label{eq:alm-stage-length}
 \left\lceil\sqrt{\frac{2\beta}{2\beta-1}
 \left(1+\frac{2L}{\rho}\right)}\right\rceil
\end{equation}
}
APBM-One-Step calls.  Under~\eqref{eq:alm-holder-model} and $B\geq1$,
the number of APBM-One-Step calls per stage is
{\small
\begin{equation}\label{eq:alm-universal-stage}
 O\!\left(1+
 \left(\frac{M_\nu^2}{\rho^{1+\nu}\Delta^{1-\nu}}
 \right)^{\frac1{1+3\nu}}\right),\qquad 0\leq\nu\leq1.
\end{equation}
}
The hidden constant depends only on $\beta$.
\end{lem}

\begin{proof}
If neither stopping condition holds after call $t$, then
$f(\hat{x}^{t+1})-\ell>(2\beta-1)\Delta>0$ by the definition of $\ell$,
and $0\leq\Phi_t-\Phi_0\leq\beta\Delta$ by
Lemma~\ref{lem:alm-model-growth} and $\Phi_0\geq0$.

\emph{Piecewise smoothness.}
By Proposition~\ref{prop:alm-block-progress},
\eqref{eq:alm-accumulated-progress} and $\omega_t=t(t+1)/2$ give
$t(t+1)(2\beta-1)/2<\beta(1+2L/\rho)$.
Since $t(t+1)>t^2$, this inequality fails when $t$ reaches
\eqref{eq:alm-stage-length}, proving the first bound.

\emph{H\"older smoothness.}
Using~\eqref{eq:alm-warm-progress} and $\omega_t\geq t^2/2$ instead gives
{\small
\begin{equation*}
 t^2(2\beta-1)\Delta<2\beta\Delta
 +4M_\nu\left(\frac{\beta\Delta}{\rho}\right)^{(1+\nu)/2}
 t^{3(1-\nu)/2}.
\end{equation*}
}
If $t^2\geq4\beta/(2\beta-1)$, then
$2\beta\Delta\leq t^2(2\beta-1)\Delta/2$.
Moving this term to the left and rearranging yields
$t<\left(\frac{64\beta^{1+\nu}M_\nu^2}
 {(2\beta-1)^2\rho^{1+\nu}\Delta^{1-\nu}}
 \right)^{\frac1{1+3\nu}}.$
For smaller $t$, we have $t<2\sqrt{\beta/(2\beta-1)}$.
The prefactors in these two bounds are bounded uniformly over
$\nu\in[0,1]$ for fixed $\beta$, proving~\eqref{eq:alm-universal-stage}
with the stated constant dependence.
\end{proof}

We are now ready to prove the oracle complexity bounds for
Alg.~\ref{alg:accelerated-lm-pbm}.
\begin{proof}[\underline{Proof of Theorem~\ref{thm:alm-complexity}}]
Completing the square in~\eqref{eq:rbcer} shows that the initial oracle cut
gives a $(\rho/2,\Delta_1/2)$-RBCer.  Since $\rho\leq\mu^*$,
Proposition~\ref{prop:rbcer-objective-gap} implies
$f(\bar{y}^{1})-f^*\leq\Delta_1$, so $\underline f_1\leq f^*$.

Consider a stage with $\Delta_s>\vep$.
Since $\hat{x}^{1}=\bar{y}^{s}$, $\lambda_0=0$, and $\ell_s<f(\bar{y}^{s})$,
Lemma~\ref{lem:alm-model-growth} gives $a^t\leq f$ and $\lambda_t>0$.
Thus the certificate computation in line~\ref{line:alm-evaluate-certificate}
of Alg.~\ref{alg:accelerated-lm-pbm} is well defined.
If the certificate condition in line~\ref{line:lower_bound_improve}
of Alg.~\ref{alg:accelerated-lm-pbm} holds,
Proposition~\ref{prop:alm-model-test} gives $\ell_s\leq f^*$.
The update $\underline f_{s+1}=\ell_s$ then yields
$\Delta_{s+1}\leq f(\bar{y}^{s})-\ell_s=\beta\Delta_s$,
because line~\ref{line:alm-best-point} of Alg.~\ref{alg:alm-block}
ensures $f(\bar{y}^{s+1})\leq f(\bar{y}^{s})$.
Otherwise, the descent condition holds and the lower bound is unchanged, giving
$\Delta_{s+1}\leq\Delta_s-(1-\beta)\Delta_s=\beta\Delta_s$.
Both updates preserve $\underline f_{s+1}\leq f^*$.
Induction therefore gives $f(\bar{y}^{s})-f^*\leq\Delta_s$ and
$\Delta_{s+1}\leq\beta\Delta_s$.

Each stage starts with $\Phi_0=0$, and~\eqref{eq:alm-certificate-identity}
makes its certificate condition equivalent to $\Phi_t>\beta\Delta_s$.
Lemma~\ref{lem:alm-stage} therefore bounds the number of APBM-One-Step calls
per stage by~\eqref{eq:alm-stage-length}.
Let $S$ be the number of stages needed to reach $\Delta_{S+1}\leq\vep$.
The contraction gives
$S\leq\lceil\log_+(\Delta_1/\vep)/\log(1/\beta)\rceil$.
The point $\bar{y}^{S+1}$ is therefore $\vep$-optimal, since
$f(\bar{y}^{S+1})-f^*\leq\Delta_{S+1}\leq\vep$.
Each call uses at most $B+1\leq2B$ oracle queries.
Multiplying these bounds, adding the initial query, and using $\rho=\mu^*$ proves
\eqref{eq:alm-complexity}.
\end{proof}

\begin{proof}[\underline{Proof of Theorem~\ref{thm:alm-universal}}]
As in the proof of Theorem~\ref{thm:alm-complexity}, initialization and
the lower-bound updates give $f(\bar{y}^{s})-f^*\leq\Delta_s$ and
$\Delta_{s+1}\leq\beta\Delta_s$ under~\eqref{eq:alm-holder-model}.
If $\Delta_1\leq\vep$, the initial query suffices.
Otherwise, let $S$ be the number of stages performed, so
$\Delta_S>\vep$ and
$S\leq\lceil\log_+(\Delta_1/\vep)/\log(1/\beta)\rceil$.
Iterating $\Delta_{s+1}\leq\beta\Delta_s$ gives
$\Delta_s\geq\beta^{-(S-s)}\Delta_S>\beta^{-(S-s)}\vep$
for $1\leq s\leq S$.  Since $r(\nu)\geq0$,
$
 \sum_{s=1}^{S}\Delta_s^{-r(\nu)}
 \leq\vep^{-r(\nu)}\sum_{j=0}^{S-1}\beta^{jr(\nu)}.
$
For $0<\nu<1$, the finite geometric sum equals
$(1-\beta^{Sr(\nu)})/(1-\beta^{r(\nu)})$ and is bounded by
$\min\{S,(1-\beta^{r(\nu)})^{-1}\}$.
Moreover, $S\leq\max\{1,1/\log(1/\beta)\}
[1+\log_+(\Delta_1/\vep)]$.
Substituting this bound for $S$ in both the stage count and the minimum,
and summing~\eqref{eq:alm-universal-stage} in Lemma~\ref{lem:alm-stage},
gives~\eqref{eq:alm-universal-complexity} with a constant depending only
on $\beta$, since each call uses at most $B+1\leq2B$ oracle queries
and initialization uses one query.

For $\nu=1$, Lemma~\ref{lem:alm-stage} with $\rho=\mu^*$ gives
$O(1+\sqrt{M_1/\mu^*})$ APBM-One-Step calls per stage.
The $S$ stages therefore require $O(BS(1+\sqrt{M_1/\mu^*}))$
first-order oracle calls, proving the smooth case.
For $\nu=0$, we have $r(0)=1$ and
$\sum_{j=0}^{S-1}\beta^j\leq(1-\beta)^{-1}$.
Summing~\eqref{eq:alm-universal-stage} therefore gives
$O(B[1+\log_+(\Delta_1/\vep)+M_0^2/(\mu^*\vep)])$
first-order oracle calls, proving the nonsmooth case.
\end{proof}

Under the parameter choice in Theorem~\ref{thm:alm-complexity},
Alg.~\ref{alg:accelerated-lm-pbm} satisfies
$f(\bar{y}^{s})-f^*\leq\Delta_s=f(\bar{y}^{s})-\underline f_s$,
so $\Delta_s\leq\vep$ certifies $\vep$-optimality.
In Alg.~\ref{alg:adaptive-apbm}, however, the validity of the lower bound
depends on the parameter estimate.

By~\eqref{eq:alm-certificate-identity}, the certificate condition in
line~\ref{line:apbm-certify} of Alg.~\ref{alg:adaptive-apbm}
is equivalent to $\Phi_t\geq2\Delta$, since $\ell=f(\bar{y})-\Delta$.
Applying Proposition~\ref{prop:alm-model-test} with parameters
$\rho/2$ and $\lambda/2$ gives
$f(\bar{y})-f^*\leq\max\{1,\rho/(2\mu^*)\}\Delta$ under QG.
At initialization, completing the square in~\eqref{eq:rbcer} and applying
Proposition~\ref{prop:rbcer-objective-gap} with proximal parameter
$\widehat\rho_1/4$ give the same estimate with
$(\bar{y},\rho,\Delta)=(\bar{y}^{1},\widehat\rho_1,\Delta_1)$.

When a queried objective value falls below the recorded lower bound,
line~\ref{line:apbm-update} of Alg.~\ref{alg:adaptive-apbm} updates
$(\rho,\Delta)\gets(\rho/2,2\Delta)$ while retaining $(\bar{y},a,\lambda)$.
Lemma~\ref{lem:alm-parameter-transfer} shows that this update does not
decrease $\Phi/\Delta$, preserving progress toward certification as long
as the reference stays unchanged.

\begin{lem}\label{lem:alm-parameter-transfer}
For any lower semicontinuous convex $\psi\leq f$,
$\bar{y}\in X$, and $\rho>0$,
{\small
\begin{equation}\label{eq:alm-rbcer-scaling}
 \Delta_{\rho/2}(\bar{y};\psi)
 \leq2\Delta_\rho(\bar{y};\psi)-[f(\bar{y})-\psi(\bar{y})]
 \leq2\Delta_\rho(\bar{y};\psi).
\end{equation}
}
For affine $a\leq f$, $\lambda\geq0$, and $\Delta>0$, define
$\Phi_{\rm old}$ by~\eqref{eq:alm-model-minimum} with
$\ell=f(\bar{y})-\Delta$.  Then
{\small
\begin{equation}\label{eq:alm-parameter-transfer}
 \Phi_{\rm new}:=\min_{x\in X}
 \left\{\lambda\bigl(a(x)-f(\bar{y})+2\Delta\bigr)
 +\frac\rho4\norm{x-\bar{y}}^2\right\}
 \geq2\Phi_{\rm old}.
\end{equation}
}
\end{lem}

\begin{proof}
For any $u\in X$, set $w=(\bar{y}+u)/2\in X$.
Convexity gives $\psi(u)\geq2\psi(w)-\psi(\bar{y})$, and
$\norm{u-\bar{y}}^2=4\norm{w-\bar{y}}^2$.  Therefore
{\small
\begin{equation*}
 f(\bar{y})-\psi(u)-\frac\rho4\norm{u-\bar{y}}^2
 \leq2\Delta_\rho(\bar{y};\psi)-[f(\bar{y})-\psi(\bar{y})].
\end{equation*}
}
Taking the supremum over $u$ proves~\eqref{eq:alm-rbcer-scaling},
including its last inequality since $\psi(\bar{y})\leq f(\bar{y})$.
For $\lambda>0$, apply it with $\psi=a$ and parameter $\rho/\lambda$:
{\small
\begin{equation*}
 \Phi_{\rm new}
 =\lambda[2\Delta-\Delta_{\rho/(2\lambda)}(\bar{y};a)]
 \geq2\Phi_{\rm old}+\lambda[f(\bar{y})-a(\bar{y})]
 \geq2\Phi_{\rm old}.
\end{equation*}
}
For $\lambda=0$, both minima are zero.
\end{proof}

We next bound the number of parameter halvings.
\begin{lem}\label{lem:alm-halving}
In Alg.~\ref{alg:adaptive-apbm}, let $Q$ denote the total number of
parameter halvings and $Q_s$ the number in outer iteration $s$.
Under QG, each halving requires $\rho>\mu^*$ and yields an updated
$\rho>\mu^*/2$.  Consequently,
{\small
\begin{equation*}
 Q=\sum_{s\geq1}Q_s
 \leq\left\lceil\log_2\frac{\widehat\rho_1}{\mu^*}\right\rceil_+.
\end{equation*}
}
\end{lem}

\begin{proof}
After $j$ halvings in outer iteration $s$,
$\rho=2^{-j}\widehat\rho_s$ and
$\Delta=2^j\beta\Delta_s$.
Substituting these identities into the objective-gap estimate above gives
$f(\bar{y}^{s})-f^*\leq
 \max\{\Delta_s,\rho\Delta/(2\beta\mu^*)\}.$
Hence $f(\bar{y}^{s})-\max\{\Delta_s,\Delta/(2\beta)\}\leq f^*$
whenever $\rho\leq\mu^*$.
Line~\ref{line:apbm-lower-bound} of Alg.~\ref{alg:adaptive-apbm}
takes the maximum of these bounds since initialization or the last halving.
Since $\rho$ is unchanged over this interval and the record starts at
$-\infty$, their maximum also satisfies $\underline f\leq f^*$ whenever
$\rho\leq\mu^*$.  The condition $f(\hat{x})<\underline f$ in
line~\ref{line:apbm-update} of Alg.~\ref{alg:adaptive-apbm} can therefore
hold only if $\rho>\mu^*$, so the updated parameter exceeds $\mu^*/2$.
Since $\rho$ starts at $\widehat\rho_1$ and changes only by halving,
the bound on $Q$ follows.
\end{proof}

We next bound the total cost of parameter search.
The following lemma uses the certificate progress retained across halvings
at the same reference point to obtain an additive search cost for all
H\"older exponents $0\leq\nu\leq1$.

\begin{lem}\label{lem:alm-search-progress}
Consider $J+1$ consecutive sequences of APBM-One-Step calls in
Alg.~\ref{alg:adaptive-apbm}, each starting at $t=1$ and ending at a branch in
lines~\ref{line:apbm-update}--\ref{line:apbm-reference}.
Suppose $\bar{y}$ is unchanged, $\Phi_0\geq0$ initially, and
$(\rho,\Delta)\gets(\rho/2,2\Delta)$ between sequences.
Under~\eqref{eq:alm-holder-model} with $B\geq1$, the total number of
APBM-One-Step calls is
{\small
\begin{equation}\label{eq:alm-warm-search-cost}
 O\!\left(J+1+
 \left(\frac{M_\nu^2}{\rho_f^{1+\nu}\Delta_f^{1-\nu}}
 \right)^{\frac1{1+3\nu}}\right),\qquad 0\leq\nu\leq1.
\end{equation}
}
Here $(\rho_f,\Delta_f)$ are the parameters in the last sequence.
\end{lem}

\begin{proof}
Index the sequences backward by $i=0,\ldots,J$, so
$\Delta_i=2^{-i}\Delta_f$ and $\rho_i\Delta_i=\rho_f\Delta_f$.
For sequence $i$, let $T_i$ count its APBM-One-Step calls, write
$\Phi_t^{(i)}$ for its subproblem minimum, and set
$e_i=(\Phi_{T_i-1}^{(i)}-\Phi_0^{(i)})/\Delta_i$.
Lemmas~\ref{lem:alm-model-growth} and~\ref{lem:alm-parameter-transfer}
give $e_i\geq0$ and, since sequence $i-1$ follows sequence $i$,
{\small
\begin{equation*}
 \frac{\Phi_0^{(i-1)}}{\Delta_{i-1}}
 \geq\frac{\Phi_{T_i}^{(i)}}{\Delta_i}
 \geq\frac{\Phi_{T_i-1}^{(i)}}{\Delta_i},\qquad i=1,\ldots,J.
\end{equation*}
}
If every $T_i=1$, then $\sum_i e_i=0$.
Otherwise, let $j=\min\{i:T_i>1\}$, so sequence $j$ is the last with
more than one call and $e_i=0$ for $i<j$.
None of the three branches in lines~\ref{line:apbm-update}--\ref{line:apbm-reference}
of Alg.~\ref{alg:adaptive-apbm} is taken after call $T_j-1$.
In particular, the certificate condition fails, giving
$\Phi_{T_j-1}^{(j)}<2\Delta_j$.
Summing the transfer inequalities through sequence $j$ yields
{\small
\begin{equation}\label{eq:alm-dual-transfer}
 \sum_{i=0}^J e_i=\sum_{i=j}^J e_i
 \leq\frac{\Phi_{T_j-1}^{(j)}}{\Delta_j}
      -\frac{\Phi_0^{(J)}}{\Delta_J}\leq2.
\end{equation}
}

Let $q=(1+3\nu)/(1+\nu)\in[1,2]$.
At $t=T_i-1\geq1$, the stopping tests give
$f(\hat{x}^{t+1})-\ell\geq(1-1/(2\beta))\Delta_i$ and
$\Phi_t^{(i)}-\Phi_0^{(i)}=e_i\Delta_i<2\Delta_i$.
Applying~\eqref{eq:alm-warm-progress} and absorbing its linear term
as in the proof of Lemma~\ref{lem:alm-stage} yields
{\small
\begin{equation}\label{eq:alm-search-increment}
 T_i\leq C_\beta\left[1+
 \left(\frac{M_\nu^2}{\rho_f^{1+\nu}\Delta_f^{1-\nu}}
 \right)^{\frac1{1+3\nu}}
 2^{-i(q-1)/q}e_i^{1/q}\right],
\end{equation}
}
where $C_\beta$ depends only on $\beta$ and covers the bounded values of $t$.
For $M_\nu=0$, the same progress estimate gives $T_i\leq C_\beta$.
For $q>1$, H\"older's inequality and~\eqref{eq:alm-dual-transfer} give
{\small
\begin{equation}\label{eq:alm-search-moment}
 \sum_{i=0}^J2^{-i(q-1)/q}e_i^{1/q}
 \leq\left(\sum_{i=0}^J e_i\right)^{1/q}
 \left(\sum_{i=0}^J2^{-i}\right)^{(q-1)/q}\leq2.
\end{equation}
}
For $q=1$, the same bound is~\eqref{eq:alm-dual-transfer}.
Summing~\eqref{eq:alm-search-increment} now proves~\eqref{eq:alm-warm-search-cost}.
\end{proof}
Now we can prove the oracle complexity of Alg.~\ref{alg:adaptive-apbm}.
\begin{proof}[\underline{Proof of Theorem~\ref{thm:alm-adaptive}}]
If $f^{\prime}(\bar{y}^{1})=0$, convexity implies that $\bar{y}^{1}$ is optimal.
Otherwise, Lemmas~\ref{lem:alm-model-growth} and~\ref{lem:alm-parameter-transfer}
give $\Phi_0\geq0$ at initialization and after the updates in
lines~\ref{line:apbm-update} and~\ref{line:apbm-reference} of
Alg.~\ref{alg:adaptive-apbm}.  Consider the resulting sequences of
APBM-One-Step calls defined in Lemma~\ref{lem:alm-search-progress}.
Under PWS, \eqref{eq:alm-accumulated-progress} gives
$O(B(1+\sqrt{L/\rho}))$ oracle calls per sequence,
as in the proof of Lemma~\ref{lem:alm-stage}.
Under~\eqref{eq:alm-holder-model}, each sequence is finite by
Lemma~\ref{lem:alm-search-progress} with $J=0$.

In outer iteration $s$, a reference update requires both
$f(\hat{x})<f(\bar{y})-\Delta/(2\beta)$ and $f(\hat{x})\geq\underline f$.
Before the first halving, $\Delta=\beta\Delta_s$.
Two reference updates would therefore give
$f(\hat{x})<f(\bar{y}^{s})-\Delta_s\leq\underline f$, a contradiction.
After the first halving, $\Delta\geq2\beta\Delta_s$ and
$f(\bar{y})\leq f(\bar{y}^{s})$ imply
$f(\hat{x})<f(\bar{y})-\Delta/(2\beta)
\leq f(\bar{y}^{s})-\Delta/(2\beta)\leq\underline f$.
Hence line~\ref{line:apbm-reference} of Alg.~\ref{alg:adaptive-apbm}
updates $\bar{y}$ at most once, before any halving.
Lemma~\ref{lem:alm-halving} therefore ensures termination of each outer
iteration, with at most $Q_s+2$ sequences in iteration $s$.

Lemma~\ref{lem:alm-halving} and $\widehat\rho_1\geq\mu^*$ give
$\widehat\rho_s>\mu^*/2$.
Halving preserves $\rho\Delta$, so each outer iteration satisfies
$\widehat\rho_{s+1}\Delta_{s+1}=\beta\widehat\rho_s\Delta_s$.
The initialization and the conditional gap estimate above therefore give
{\small
\begin{equation}\label{eq:alm-adaptive-product}
 f(\bar{y}^{s})-f^*\leq\frac2{\mu^*}\widehat\rho_s\Delta_s
 =\frac{8\norm{f^{\prime}(\bar{y}^{1})}^2}{\mu^*}\beta^{s-1},\qquad s\geq1.
\end{equation}
}
Let $S$ be the smallest positive index satisfying
$\widehat\rho_S\Delta_S\leq\mu^*\vep/2$.
By~\eqref{eq:alm-adaptive-product}, $\bar{y}^{S}$ is $\vep$-optimal and
$S=O(1+\log_+(\norm{f^{\prime}(\bar{y}^{1})}^2/(\mu^*\vep)))$.
If $S=1$, only the initial oracle query is needed, which accounts for the
additive $1$ in~\eqref{eq:alm-adaptive-universal}. For $S>1$, minimality gives
$\beta\mu^*\vep/2<\widehat\rho_S\Delta_S\leq\mu^*\vep/2$.

Under~\eqref{eq:alm-holder-model}, apply Lemma~\ref{lem:alm-search-progress}
separately before and after the possible reference update in iteration $s$.
Since $\rho\Delta=\widehat\rho_{s+1}\Delta_{s+1}$ and
$\rho\geq\widehat\rho_{s+1}$, we have
$\rho^{1+\nu}\Delta^{1-\nu}
=\rho^{2\nu}(\rho\Delta)^{1-\nu}
\geq\widehat\rho_{s+1}^{1+\nu}\Delta_{s+1}^{1-\nu}$.
Summing over the $S-1$ completed outer iterations and indexing their outputs
from $2$ to $S$ gives the following bound, including the initial query and
at most $B+1\leq2B$ queries per APBM-One-Step call.
{\small
\begin{equation}\label{eq:alm-adaptive-accounting}
 O\!\left(B\left[S+Q+
 \sum_{s=2}^S\left(\frac{M_\nu^2}
 {\widehat\rho_s^{1+\nu}\Delta_s^{1-\nu}}\right)^{\frac1{1+3\nu}}
 \right]\right).
\end{equation}
}
For $0\leq\nu<1$, the identity
$\widehat\rho_s\Delta_s=\beta^{s-S}\widehat\rho_S\Delta_S$
and the preceding lower bounds give
{\small
\begin{equation*}
 \left(\frac{M_\nu^2}
 {\widehat\rho_s^{1+\nu}\Delta_s^{1-\nu}}\right)^{\frac1{1+3\nu}}
 \leq C_\beta
 \left(\frac{M_\nu^2}{(\mu^*)^{1+\nu}\vep^{1-\nu}}\right)^{\frac1{1+3\nu}}
 \beta^{(S-s)r(\nu)},
\end{equation*}
}
where $C_\beta$ depends only on $\beta$.
Summing over $s$, using
$\sum_{s=2}^S\beta^{(S-s)r(\nu)}\leq
\min\{S-1,(1-\beta^{r(\nu)})^{-1}\}$,
and substituting the bounds on $S$ and $Q$
proves~\eqref{eq:alm-adaptive-universal}.
For $\nu=0$, $r(0)=1$ and $\norm{f^{\prime}(\bar{y}^{1})}\leq M_0/2$ imply
$S=O(1+M_0^2/(\mu^*\vep))$; substitution gives the stated nonsmooth complexity.

Under PWS, summing the estimates for each sequence over the halvings and
the possible reference update gives
$O(B[1+Q_s+\sqrt{L/\widehat\rho_{s+1}}])$ oracle calls in iteration $s$,
since $\sum_{j=0}^{Q_s}2^{-j/2}\leq(1-2^{-1/2})^{-1}$.
Summing over $s=1,\ldots,S-1$ and using $\widehat\rho_{s+1}>\mu^*/2$
proves~\eqref{eq:alm-adaptive-pws}.
For $\nu=1$, \eqref{eq:alm-warm-progress} reduces to
\eqref{eq:alm-accumulated-progress} with $L=M_1$ and $B\geq1$,
so the same estimate applies.
\end{proof}

\section{Numerical Study}\label{sec:pbm-numerical-study}
We present numerical experiments to illustrate our theoretical findings.
Subsection~\ref{subsec:pbm-condition-number} examines condition-number dependence,
Subsection~\ref{subsec:pbm-adaptive-maxquad} studies adaptation to piecewise
smoothness and sensitivity to initialization, and
Subsection~\ref{subsec:pbm-holder-universality} evaluates universal acceleration
under H\"older smoothness.

\subsection{Condition-number dependence in Theorems~\ref{thm:lm-rbcer-complexity} and~\ref{thm:alm-complexity}}
\label{subsec:pbm-condition-number}
We use a PWS problem to examine the condition-number dependence of PBMs and APBMs.

\textbf{Experimental setup.}
Let $d=1000$, let $T_d$ be the tridiagonal matrix with diagonal entries
$2$ and adjacent off-diagonal entries $-1$, and denote its smallest and
largest eigenvalues by $\tau_{\min}$ and $\tau_{\max}$. For $\kappa>1$, set
{\small
\[
 H_\kappa=I+\frac{\kappa-1}{\tau_{\max}-\tau_{\min}}
 (T_d-\tau_{\min}I),\qquad
 b_\kappa=\left(\frac{2}{e_1^\top H_\kappa^{-1}e_1}\right)^{1/2}e_1,
\]
}
where $e_1=(1,0,\ldots,0)^\top\in\mathbb R^d$. We consider the following problem: 
{\small
\begin{equation}\label{eq:pbm-condition-instance}
 \min_{(u,z)\in\mathbb R^d\times\mathbb R}
 f_\kappa(u,z):=
 \frac12 u^\top H_\kappa u-b_\kappa^\top u+1+|z|+\frac12z^2.
\end{equation}
}
This problem has $k=2$ smooth pieces with Hessian
$\operatorname{diag}(H_\kappa,1)$, minimizer
$(H_\kappa^{-1}b_\kappa,0)$, and $f^*=0$.
The extreme eigenvalues of $H_\kappa$ give $L=\kappa$ and $\mu^*=1$.
The growth bound is attained along a minimum-eigenvalue direction with $z=0$.
Both methods start at $(0,0)$ with objective gap $1$.
Algs.~\ref{alg:lm-pbm} and~\ref{alg:accelerated-lm-pbm} use
$B=2$ and $\rho=\mu^*=1$, with $\beta=0.75$.

% Theorems~\ref{thm:lm-rbcer-complexity} and~\ref{thm:alm-complexity}
% predict linear and square-root dependence on $\kappa$ for PBMs and
% APBMs, respectively, apart from initialization and accuracy logarithms.

\textbf{Observations.}
Fig.~\ref{fig:pbm-condition-number} shows that both methods reach the
target accuracy throughout the tested range. PBMs require fewer calls
on the better-conditioned instances, whereas the advantage of APBMs
increases with ill-conditioning. On the ill-conditioned instances,
the oracle costs follow approximately linear and square-root trends,
consistent with the leading condition-number dependence in
Theorems~\ref{thm:lm-rbcer-complexity} and~\ref{thm:alm-complexity}, respectively.

\subsection{Adaptive linear convergence and sensitivity to initialization}
\label{subsec:pbm-adaptive-maxquad}
We use MAXQUAD to examine adaptation to PWS structure and sensitivity to the
initial estimate $\widehat\rho_1$ in Alg.~\ref{alg:adaptive-apbm}.
The experimental results use zero-based indices, so $\widehat\rho_0$
in the table and figure denotes this initial estimate.

\textbf{Experimental setup.}
We consider
{\small
\begin{equation}\label{eq:pbm-random-maxquad}
 \min_{x\in\mathbb R^n} f(x):=
 \max_{1\leq i\leq k}\left\{f_i(x)=
 \frac12x^\top A_i x+b_i^\top x+c_i\right\}.
\end{equation}
}
Each $A_i$ has an independent random orthogonal eigenbasis and
eigenvalues evenly spaced from $1$ to $L$. The entries of $b_i$ and
$c_i$ are independent standard normal samples. Thus, $f$ is
$(k,L)$-PWS and $1$-strongly convex, with $1\leq\mu^*\leq L$.

We set $n=300$, $B=20$,
$\beta=0.75$, and $x^0=0$. We vary $(k,L,\widehat\rho_0)$ as listed in
Tab.~\ref{tab:pbm-numerical-summary}, using the same instance for each
$(k,L)$ across initial estimates. We stop only after
Alg.~\ref{alg:adaptive-apbm} executes the certification branch in
line~\ref{line:apbm-certify} and the resulting stage satisfies
$\Delta_s\leq10^{-6}$. By Proposition~\ref{prop:alm-model-test},
this bounds the optimality gap
by $\max\{1,\widehat\rho_s/(2\mu^*)\}10^{-6}$.
We also independently check the returned accuracy.
Each run is capped at $500{,}000$ oracle calls.
We compute the reference value $f_{\rm ref}$ and point $x_{\rm ref}$ using
Ipopt~\cite{wachter2006implementation}.
We define the number of numerically active components by
$
 k^*:=\left|\left\{i\in[k]:
 f(x_{\rm ref})-f_i(x_{\rm ref})\leq10^{-8}\right\}\right|.
$

\textbf{Observations.}
Tab.~\ref{tab:pbm-numerical-summary} shows that all runs with $B\geq k$
meet both the internal and independently checked accuracy targets.
When $B<k$, runs with $k^*\leq B$ also reach both targets,
whereas those with $k^*>B$ exhaust the oracle budget.
This suggests that the method can adapt to fewer locally active components. The global
guarantee in Theorem~\ref{thm:alm-adaptive} still requires $B\geq k$.
For each successful instance, oracle costs remain comparable across
the tested initial estimates.

On the plotted instance, Fig.~\ref{fig:pbm-adaptive-bounds} shows
approximately geometric decay of the upper errors and the magnitudes
of the lower errors after initial adjustments, illustrating adaptive
linear convergence. The upper curves include reference-point updates
within stages. The conditional lower bound is nondecreasing while
$\rho$ remains fixed, but may decrease when a queried value falls below
it and triggers halving. The final estimates need not equal $\mu^*$.
\begin{figure}[tbp]
\setlength{\abovecaptionskip}{0pt}
\centering
\begin{minipage}[t]{0.485\linewidth}
  \centering
  \includegraphics[width=\linewidth]{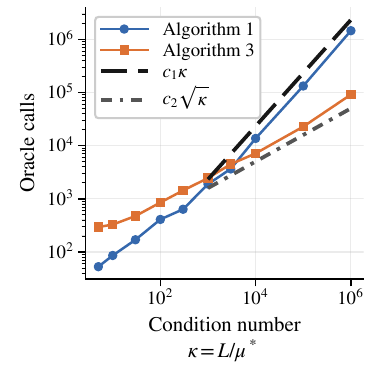}
  \vspace{-14pt}
  \caption{Oracle calls versus condition number
  for~\eqref{eq:pbm-condition-instance}.}
  \label{fig:pbm-condition-number}
\end{minipage}\hfill
\begin{minipage}[t]{0.485\linewidth}
  \centering
  \includegraphics[width=\linewidth]{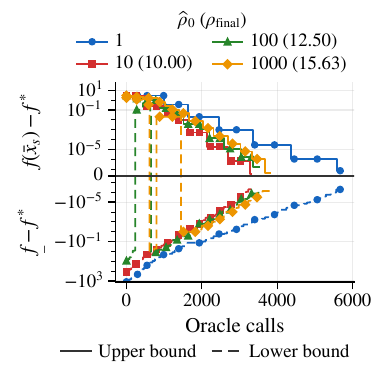}
  \vspace{-14pt}
  \caption{Upper and lower errors for APBMs on MAXQUAD
  ($n=300$, $k=B=20$, $L=10$). Parentheses give the final $\rho$.}
  \label{fig:pbm-adaptive-bounds}
\end{minipage}
\end{figure}

% Data source: src_pbm/results/maxquad_full_line9_20261024/cases/assembled/processed/table2.csv
% Source SHA-256: 5d291e7edbcd7a296e828a6627ac88bd55c99004bb6d7090d909f350b8657f80
% Fresh 32-case rerun: 2026-10-24; successful stops require Algorithm 4 line 9 and Delta_s <= 1e-6.
% Generated from audited raw runs by maxquad_table2/export.py.
\begin{table}[tbp]
\centering
\scriptsize
\caption{Alg.~\ref{alg:adaptive-apbm} on MAXQUAD ($n=300$, $B=20$). Superscript C marks the oracle limit.}
\label{tab:pbm-numerical-summary}
\setlength{\tabcolsep}{2pt}
\renewcommand{\arraystretch}{0.9}
\setlength{\aboverulesep}{1.2pt}
\setlength{\belowrulesep}{1.5pt}
\begin{tabular*}{0.80\linewidth}{@{\extracolsep{\fill}}rrrrrrrr@{}}
\toprule
$k$ & $L$ & $k^*$ & $\widehat\rho_0$ & $\rho_{\rm final}$ & Oracle calls & $f(\bar x_s)-f_{\rm ref}$ & $\Delta_s$\\
\midrule
20 & 10 & 20 & 1 & 1.00 & $5,755$ & $7.80\times10^{-8}$ & $8.99\times10^{-7}$\\
 &  &  & 10 & 10.00 & $3,319$ & $2.89\times10^{-8}$ & $8.98\times10^{-7}$\\
 &  &  & 100 & 12.50 & $3,529$ & $1.72\times10^{-7}$ & $9.58\times10^{-7}$\\
 &  &  & 1000 & 15.63 & $3,802$ & $4.32\times10^{-8}$ & $7.66\times10^{-7}$\\
\midrule
20 & 100 & 11 & 1 & 1.00 & $4,894$ & $1.56\times10^{-7}$ & $8.99\times10^{-7}$\\
 &  &  & 10 & 10.00 & $3,193$ & $7.28\times10^{-8}$ & $8.98\times10^{-7}$\\
 &  &  & 100 & 100.00 & $3,025$ & $1.11\times10^{-7}$ & $8.97\times10^{-7}$\\
 &  &  & 1000 & 62.50 & $3,361$ & $3.79\times10^{-7}$ & $8.07\times10^{-7}$\\
\midrule
20 & 1000 & 2 & 1 & 1.00 & $6,364$ & $2.16\times10^{-7}$ & $8.99\times10^{-7}$\\
 &  &  & 10 & 10.00 & $4,033$ & $5.31\times10^{-7}$ & $8.98\times10^{-7}$\\
 &  &  & 100 & 100.00 & $2,710$ & $1.22\times10^{-7}$ & $8.97\times10^{-7}$\\
 &  &  & 1000 & 250.00 & $2,962$ & $3.31\times10^{-7}$ & $8.51\times10^{-7}$\\
\midrule
100 & 10 & 48 & 1 & 1.00 & $500,000^{\mathrm{C}}$ & $2.01\times10^{-3}$ & $4.12\times10^{-3}$\\
 &  &  & 10 & 10.00 & $500,000^{\mathrm{C}}$ & $1.59\times10^{-4}$ & $3.09\times10^{-4}$\\
 &  &  & 100 & 25.00 & $500,000^{\mathrm{C}}$ & $7.97\times10^{-5}$ & $1.24\times10^{-4}$\\
 &  &  & 1000 & 15.63 & $500,000^{\mathrm{C}}$ & $4.91\times10^{-5}$ & $1.48\times10^{-4}$\\
\midrule
100 & 100 & 13 & 1 & 1.00 & $5,839$ & $1.61\times10^{-7}$ & $9.80\times10^{-7}$\\
 &  &  & 10 & 10.00 & $3,214$ & $1.61\times10^{-7}$ & $9.79\times10^{-7}$\\
 &  &  & 100 & 100.00 & $2,899$ & $1.96\times10^{-7}$ & $9.78\times10^{-7}$\\
 &  &  & 1000 & 125.00 & $3,151$ & $4.74\times10^{-7}$ & $7.83\times10^{-7}$\\
\midrule
100 & 1000 & 5 & 1 & 1.00 & $6,301$ & $4.38\times10^{-7}$ & $9.80\times10^{-7}$\\
 &  &  & 10 & 10.00 & $4,117$ & $3.14\times10^{-7}$ & $9.79\times10^{-7}$\\
 &  &  & 100 & 100.00 & $2,752$ & $2.52\times10^{-8}$ & $9.78\times10^{-7}$\\
 &  &  & 1000 & 1000.00 & $2,668$ & $7.97\times10^{-7}$ & $9.77\times10^{-7}$\\
\bottomrule
\end{tabular*}
\end{table}

\subsection{Universal acceleration under H\"older smoothness}
\label{subsec:pbm-holder-universality}
We examine how APBM oracle costs vary with H\"older smoothness, growth, and
tolerance.

\textbf{Experimental setup.}
For $n=300$, $\mu,\eta>0$, and $\nu\in\{1/4,1/2,3/4\}$, consider
the following problem on $X=\{x\in\mathbb R^n:\norm{x}\leq2\}$:
{\small
\begin{equation}\label{eq:pbm-holder-chain}
 \min_{x\in X} f(x):=
 \frac{\mu}{2}\norm{x-e_1}^2
 +\frac{\eta}{1+\nu}\sum_{i=1}^{n-1}|x_i-x_{i+1}|^{1+\nu},
\end{equation}
}
where $e_1=(1,0,\ldots,0)^\top\in\mathbb R^n$.
The objective satisfies the $\mu$-QG condition and the H\"older
smoothness condition~\eqref{eq:alm-holder-model} with
{\small
\begin{equation}\label{eq:pbm-holder-constant}
 M_\nu=\frac{1+\nu}{2}\mu4^{1-\nu}
       +4\eta(n-1)^{(1-\nu)/2}.
\end{equation}
}
We apply Alg.~\ref{alg:accelerated-lm-pbm} with $B=1$,
$\beta=0.75$, $\rho=\mu$, and $\bar{y}^{1}=0$. 
We terminate when the update in line~\ref{line:test-condition} of
Alg.~\ref{alg:accelerated-lm-pbm} first yields $\Delta_{s+1}\leq\vep$.
Theorem~\ref{thm:alm-universal} then guarantees
$f(\bar{y}^{s+1})-f^*\leq\vep$.
The three panels of Fig.~\ref{fig:pbm-holder-sweeps} vary
$\eta\in\{2^j:j=-1,\ldots,5\}$,
$\mu\in\{2^j:j=-3,\ldots,3\}$, and $\vep$ from $10^{-2}$ to
$10^{-5}$, respectively.

We plot the full expression inside $O(\cdot)$
in~\eqref{eq:alm-universal-complexity} at each instance's parameters,
scaling each curve to match the measured cost at the middle sampled point.
These reference curves compare parameter dependence and are not numerical
upper bounds.

\begin{figure}[tbp]
 \setlength{\abovecaptionskip}{0pt}
 \centering
 \includegraphics[width=\linewidth]{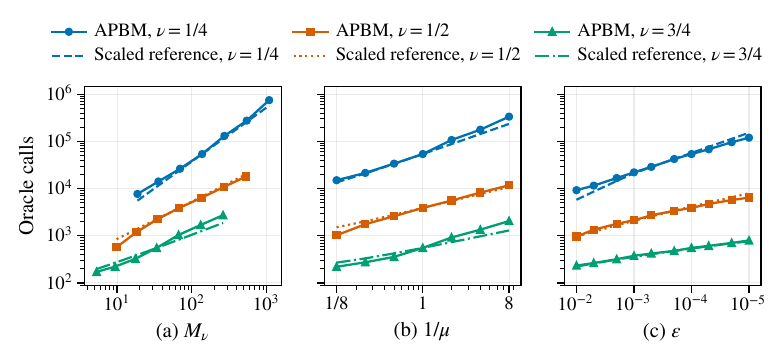}
 \vspace{-14pt}
 \caption{Oracle calls to reach $\Delta_s\leq\vep$ for~\eqref{eq:pbm-holder-chain}.
 (a) Varying $\eta$, with $\mu=1$ and $\vep=10^{-4}$; the horizontal
 axis gives $M_\nu$ from~\eqref{eq:pbm-holder-constant}.
 (b) Varying $1/\mu$, with $\eta=4$ and $\vep=10^{-4}$.
 (c) Decreasing $\vep$, with $\eta=4$ and $\mu=1$.
 Solid curves show APBM costs; the other line styles give scaled
 references based on Theorem~\ref{thm:alm-universal}.}
 \label{fig:pbm-holder-sweeps}
 % Retain the original figure labels as aliases for the combined figure.
 \label{fig:pbm-holder-M}
 \label{fig:pbm-holder-mu}
 \label{fig:pbm-holder-epsilon}
\end{figure}

\textbf{Observations.}
The measured oracle costs in Fig.~\ref{fig:pbm-holder-sweeps} follow the
overall trends of the scaled reference curves. Costs rise with larger
$M_\nu$, smaller $\mu$, and tighter tolerances. For fixed
$(n,\eta,\mu,\vep)$ and algorithm settings, smaller values of $\nu$
also require more oracle calls on the tested instances.

\FloatBarrier

\bibliographystyle{abbrvnat}
\bibliography{ref}

\begin{thebibliography}{66}
\providecommand{\natexlab}[1]{#1}
\providecommand{\url}[1]{\texttt{#1}}
\expandafter\ifx\csname urlstyle\endcsname\relax
  \providecommand{\doi}[1]{doi: #1}\else
  \providecommand{\doi}{doi: \begingroup \urlstyle{rm}\Url}\fi

\bibitem[Atenas et~al.(2023)Atenas, Sagastiz{\'a}bal, Silva, and
  Solodov]{atenas2023unified}
F.~Atenas, C.~Sagastiz{\'a}bal, P.~J. Silva, and M.~Solodov.
\newblock A unified analysis of descent sequences in weakly convex
  optimization, including convergence rates for bundle methods.
\newblock \emph{SIAM Journal on Optimization}, 33\penalty0 (1):\penalty0
  89--115, 2023.

\bibitem[Bacaud et~al.(2001)Bacaud, Lemar{\'e}chal, Renaud, and
  Sagastiz{\'a}bal]{bacaud2001bundle}
L.~Bacaud, C.~Lemar{\'e}chal, A.~Renaud, and C.~Sagastiz{\'a}bal.
\newblock Bundle methods in stochastic optimal power management: A
  disaggregated approach using preconditioners.
\newblock \emph{Computational Optimization and Applications}, 20\penalty0
  (3):\penalty0 227--244, 2001.
\newblock \doi{10.1023/A:1011202900805}.

\bibitem[Bagirov et~al.(2014)Bagirov, Karmitsa, and
  M{\"a}kel{\"a}]{bagirov2014introduction}
A.~Bagirov, N.~Karmitsa, and M.~M. M{\"a}kel{\"a}.
\newblock \emph{Introduction to Nonsmooth Optimization: theory, practice and
  software}, volume~12.
\newblock Springer, 2014.

\bibitem[Ben-Tal and Nemirovski(2005)]{ben2005non}
A.~Ben-Tal and A.~Nemirovski.
\newblock Non-euclidean restricted memory level method for large-scale convex
  optimization.
\newblock \emph{Mathematical Programming}, 102\penalty0 (3):\penalty0 407--456,
  2005.

\bibitem[Bonnans et~al.(1995)Bonnans, Gilbert, Lemar{\'e}chal, and
  Sagastiz{\'a}bal]{bonnans1995variable}
J.-F. Bonnans, J.-C. Gilbert, C.~Lemar{\'e}chal, and C.~A. Sagastiz{\'a}bal.
\newblock A family of variable metric proximal methods.
\newblock \emph{Mathematical Programming}, 68:\penalty0 15--47, 1995.

\bibitem[Bornd{\"o}rfer et~al.(2024)Bornd{\"o}rfer, L{\"o}bel, L{\"o}bel, and
  Weider]{borndorfer2024electric}
R.~Bornd{\"o}rfer, A.~L{\"o}bel, F.~L{\"o}bel, and S.~Weider.
\newblock Solving the electric bus scheduling problem by an integrated flow and
  set partitioning approach.
\newblock In \emph{24th Symposium on Algorithmic Approaches for Transportation
  Modelling, Optimization, and Systems}, volume 123 of \emph{Open Access Series
  in Informatics}, pages 11:1--11:16. Schloss Dagstuhl--Leibniz-Zentrum f{\"u}r
  Informatik, 2024.
\newblock \doi{10.4230/OASIcs.ATMOS.2024.11}.

\bibitem[Burns and Liang(2026)]{burns2026accuracy}
M.~X. Burns and J.~Liang.
\newblock Accuracy certificates for convex optimization at accelerated rates
  via primal-dual averaging, 2026.
\newblock URL \url{https://arxiv.org/abs/2604.18321v1}.

\bibitem[de~Oliveira(2019)]{de2019proximal}
W.~de~Oliveira.
\newblock Proximal bundle methods for nonsmooth dc programming.
\newblock \emph{Journal of Global Optimization}, 75\penalty0 (2):\penalty0
  523--563, 2019.

\bibitem[de~Oliveira and Sagastiz{\'a}bal(2014)]{de2014level}
W.~de~Oliveira and C.~Sagastiz{\'a}bal.
\newblock Level bundle methods for oracles with on-demand accuracy.
\newblock \emph{Optimization Methods and Software}, 29\penalty0 (6):\penalty0
  1180--1209, 2014.

\bibitem[de~Oliveira and Solodov(2016)]{oliveira2016doubly}
W.~de~Oliveira and M.~V. Solodov.
\newblock A doubly stabilized bundle method for nonsmooth convex optimization.
\newblock \emph{Mathematical Programming}, 156:\penalty0 125--159, 2016.

\bibitem[de~Oliveira et~al.(2014)de~Oliveira, Sagastiz{\'a}bal, and
  Lemar{\'e}chal]{de2014convex}
W.~de~Oliveira, C.~Sagastiz{\'a}bal, and C.~Lemar{\'e}chal.
\newblock Convex proximal bundle methods in depth: a unified analysis for
  inexact oracles.
\newblock \emph{Mathematical Programming}, 148\penalty0 (1):\penalty0 241--277,
  2014.

\bibitem[Deng et~al.(2026)Deng, Lan, and Lin]{deng2024uniformly}
Q.~Deng, G.~Lan, and Z.~Lin.
\newblock Uniformly optimal and parameter-free first-order methods for convex
  and function-constrained optimization.
\newblock \emph{INFORMS Journal on Computing}, 2026.
\newblock \doi{10.1287/ijoc.2025.1177}.

\bibitem[D{\'\i}az and Grimmer(2023)]{diaz2023optimal}
M.~D{\'\i}az and B.~Grimmer.
\newblock Optimal convergence rates for the proximal bundle method.
\newblock \emph{SIAM Journal on Optimization}, 33\penalty0 (2):\penalty0
  424--454, 2023.
\newblock \doi{10.1137/21M1428601}.

\bibitem[Doikov(2025)]{doikov2025lower}
N.~Doikov.
\newblock Lower complexity bounds for minimizing regularized functions.
\newblock \emph{Optimization Letters}, 19:\penalty0 1759--1778, 2025.
\newblock \doi{10.1007/s11590-025-02237-x}.

\bibitem[Du and Ruszczy{\'n}ski(2017)]{du2017rate}
Y.~Du and A.~Ruszczy{\'n}ski.
\newblock Rate of convergence of the bundle method.
\newblock \emph{Journal of Optimization Theory and Applications}, 173\penalty0
  (3):\penalty0 908--922, 2017.
\newblock \doi{10.1007/s10957-017-1108-1}.

\bibitem[Feltenmark and Kiwiel(2000)]{feltenmark2000dual}
S.~Feltenmark and K.~C. Kiwiel.
\newblock Dual applications of proximal bundle methods, including lagrangian
  relaxation of nonconvex problems.
\newblock \emph{SIAM Journal on Optimization}, 10\penalty0 (3):\penalty0
  697--721, 2000.
\newblock \doi{10.1137/S1052623498332336}.

\bibitem[Fersztand and Sun(2025)]{fersztand2025acceleration}
D.~Fersztand and X.~A. Sun.
\newblock On the acceleration of proximal bundle methods, 2025.
\newblock URL \url{https://arxiv.org/abs/2504.20351v1}.

\bibitem[Fersztand and Sun(2026)]{fersztand2024frankwolfe}
D.~Fersztand and X.~A. Sun.
\newblock An absolute-error proximal bundle method through the lens of
  {Frank-Wolf}, 2026.
\newblock URL \url{https://arxiv.org/abs/2411.15926v3}.

\bibitem[Fischer(2025)]{fischer2025asynchronous}
F.~Fischer.
\newblock An asynchronous proximal bundle method.
\newblock \emph{Mathematical Programming}, 209:\penalty0 825--857, 2025.
\newblock \doi{10.1007/s10107-024-02088-x}.

\bibitem[Frangioni(2002)]{frangioni2002generalized}
A.~Frangioni.
\newblock Generalized bundle methods.
\newblock \emph{SIAM Journal on Optimization}, 13\penalty0 (1):\penalty0
  117--156, 2002.

\bibitem[Frangioni(2020)]{frangioni2020standard}
A.~Frangioni.
\newblock Standard bundle methods: Untrusted models and duality.
\newblock \emph{Numerical nonsmooth optimization: state of the art algorithms},
  pages 61--116, 2020.

\bibitem[Frangioni and Gallo(1999)]{frangioni1999bundle}
A.~Frangioni and G.~Gallo.
\newblock A bundle type dual-ascent approach to linear multicommodity min-cost
  flow problems.
\newblock \emph{INFORMS Journal on Computing}, 11\penalty0 (4):\penalty0
  370--393, 1999.
\newblock \doi{10.1287/ijoc.11.4.370}.

\bibitem[Guigues et~al.(2026)Guigues, Monteiro, and Tran]{guigues2024adaptive}
V.~Guigues, R.~D.~C. Monteiro, and B.~Tran.
\newblock Complexity and numerical experiments of a new adaptive generic
  proximal bundle method, 2026.
\newblock URL \url{https://arxiv.org/abs/2410.11066v3}.

\bibitem[Haarala et~al.(2007)Haarala, Miettinen, and
  M{\"a}kel{\"a}]{haarala2007limited}
N.~Haarala, K.~Miettinen, and M.~M. M{\"a}kel{\"a}.
\newblock Globally convergent limited memory bundle method for large-scale
  nonsmooth optimization.
\newblock \emph{Mathematical Programming}, 109\penalty0 (1):\penalty0 181--205,
  2007.

\bibitem[Hare and Sagastiz{\'a}bal(2010)]{hare2010redistributed}
W.~Hare and C.~Sagastiz{\'a}bal.
\newblock A redistributed proximal bundle method for nonconvex optimization.
\newblock \emph{SIAM Journal on Optimization}, 20\penalty0 (5):\penalty0
  2442--2473, 2010.

\bibitem[Hare et~al.(2016)Hare, Sagastiz{\'a}bal, and Solodov]{hare2016inexact}
W.~Hare, C.~Sagastiz{\'a}bal, and M.~V. Solodov.
\newblock A proximal bundle method for nonsmooth nonconvex functions with
  inexact information.
\newblock \emph{Computational Optimization and Applications}, 63:\penalty0
  1--28, 2016.
\newblock \doi{10.1007/s10589-015-9762-4}.

\bibitem[Hinterm{\"u}ller(2001)]{hintermuller2001approximate}
M.~Hinterm{\"u}ller.
\newblock A proximal bundle method based on approximate subgradients.
\newblock \emph{Computational Optimization and Applications}, 20\penalty0
  (3):\penalty0 245--266, 2001.
\newblock \doi{10.1023/A:1011259017643}.

\bibitem[Jiang et~al.(2026)Jiang, Tang, and Zhang]{jiang2026optimal}
L.~Jiang, K.~Tang, and Z.~Zhang.
\newblock Optimal parameter-free first-order methods for convex optimization
  with unknown growth and smoothness, 2026.
\newblock URL \url{https://arxiv.org/abs/2607.11878v1}.

\bibitem[Joki et~al.(2017)Joki, Bagirov, Karmitsa, and
  M{\"a}kel{\"a}]{joki2017dc}
K.~Joki, A.~M. Bagirov, N.~Karmitsa, and M.~M. M{\"a}kel{\"a}.
\newblock A proximal bundle method for nonsmooth {DC} optimization utilizing
  nonconvex cutting planes.
\newblock \emph{Journal of Global Optimization}, 68\penalty0 (3):\penalty0
  501--535, 2017.
\newblock \doi{10.1007/s10898-016-0488-3}.

\bibitem[Kelley(1960)]{kelley1960cutting}
J.~E. Kelley, Jr.
\newblock The cutting-plane method for solving convex programs.
\newblock \emph{Journal of the society for Industrial and Applied Mathematics},
  8\penalty0 (4):\penalty0 703--712, 1960.

\bibitem[Kim and Dandurand(2022)]{kim2022scalable}
K.~Kim and B.~Dandurand.
\newblock Scalable branching on dual decomposition of stochastic mixed-integer
  programming problems.
\newblock \emph{Mathematical Programming Computation}, 14\penalty0
  (1):\penalty0 1--41, 2022.
\newblock \doi{10.1007/s12532-021-00212-y}.

\bibitem[Kiwiel(1983)]{kiwiel1983aggregate}
K.~C. Kiwiel.
\newblock An aggregate subgradient method for nonsmooth convex minimization.
\newblock \emph{Mathematical Programming}, 27\penalty0 (3):\penalty0 320--341,
  1983.
\newblock \doi{10.1007/BF02591907}.

\bibitem[Kiwiel(1990)]{kiwiel1990proximity}
K.~C. Kiwiel.
\newblock Proximity control in bundle methods for convex nondifferentiable
  minimization.
\newblock \emph{Mathematical Programming}, 46:\penalty0 105--122, 1990.
\newblock \doi{10.1007/BF01585731}.

\bibitem[Kiwiel(1991)]{kiwiel1991exactpenalty}
K.~C. Kiwiel.
\newblock Exact penalty functions in proximal bundle methods for constrained
  convex nondifferentiable minimization.
\newblock \emph{Mathematical Programming}, 52:\penalty0 285--302, 1991.
\newblock \doi{10.1007/BF01582892}.

\bibitem[Kiwiel(1995{\natexlab{a}})]{kiwiel1995approximations}
K.~C. Kiwiel.
\newblock Approximations in proximal bundle methods and decomposition of convex
  programs.
\newblock \emph{Journal of Optimization Theory and Applications}, 84\penalty0
  (3):\penalty0 529--548, 1995{\natexlab{a}}.
\newblock \doi{10.1007/BF02191984}.

\bibitem[Kiwiel(1995{\natexlab{b}})]{kiwiel1995proximal}
K.~C. Kiwiel.
\newblock Proximal level bundle methods for convex nondifferentiable
  optimization, saddle-point problems and variational inequalities.
\newblock \emph{Mathematical Programming}, 69\penalty0 (1):\penalty0 89--109,
  1995{\natexlab{b}}.

\bibitem[Kiwiel(1996)]{kiwiel1996restricted}
K.~C. Kiwiel.
\newblock Restricted step and {Levenberg--Marquardt} techniques in proximal
  bundle methods for nonconvex nondifferentiable optimization.
\newblock \emph{SIAM Journal on Optimization}, 6\penalty0 (1):\penalty0
  227--249, 1996.

\bibitem[Kiwiel(1999)]{kiwiel1999bregman}
K.~C. Kiwiel.
\newblock A bundle {Bregman} proximal method for convex nondifferentiable
  minimization.
\newblock \emph{Mathematical Programming}, 85\penalty0 (2):\penalty0 241--258,
  1999.
\newblock \doi{10.1007/s101070050056}.

\bibitem[Kiwiel(2000)]{kiwiel2000efficiency}
K.~C. Kiwiel.
\newblock Efficiency of proximal bundle methods.
\newblock \emph{Journal of Optimization Theory and Applications}, 104\penalty0
  (3):\penalty0 589--603, 2000.

\bibitem[Kiwiel(2006)]{kiwiel2006approximate}
K.~C. Kiwiel.
\newblock A proximal bundle method with approximate subgradient linearizations.
\newblock \emph{SIAM Journal on Optimization}, 16\penalty0 (4):\penalty0
  1007--1023, 2006.
\newblock \doi{10.1137/040603929}.

\bibitem[Lan(2015)]{lan2015bundle}
G.~Lan.
\newblock Bundle-level type methods uniformly optimal for smooth and nonsmooth
  convex optimization.
\newblock \emph{Mathematical Programming}, 149\penalty0 (1):\penalty0 1--45,
  2015.

\bibitem[Lemar{\'e}chal(1975)]{lemarechal1975extension}
C.~Lemar{\'e}chal.
\newblock An extension of {Davidon} methods to non differentiable problems.
\newblock In M.~L. Balinski and P.~Wolfe, editors, \emph{Nondifferentiable
  Optimization}, volume~3 of \emph{Mathematical Programming Studies}, pages
  95--109. North-Holland, Amsterdam, 1975.
\newblock \doi{10.1007/BFb0120700}.

\bibitem[Lemar{\'e}chal and Sagastiz{\'a}bal(1997)]{lemarechal1997variable}
C.~Lemar{\'e}chal and C.~Sagastiz{\'a}bal.
\newblock Variable metric bundle methods: From conceptual to implementable
  forms.
\newblock \emph{Mathematical Programming}, 76:\penalty0 393--410, 1997.

\bibitem[Lemar{\'e}chal et~al.(1995)Lemar{\'e}chal, Nemirovskii, and
  Nesterov]{lemarechal1995new}
C.~Lemar{\'e}chal, A.~Nemirovskii, and Y.~Nesterov.
\newblock New variants of bundle methods.
\newblock \emph{Mathematical Programming}, 69\penalty0 (1):\penalty0 111--147,
  1995.

\bibitem[Liang(2025)]{liang2024primaldual}
J.~Liang.
\newblock Primal-dual proximal bundle and conditional gradient methods for
  convex problems.
\newblock \emph{Mathematical Programming}, 2025.
\newblock \doi{10.1007/s10107-025-02307-z}.

\bibitem[Liang and Monteiro(2021)]{liang2021proximal}
J.~Liang and R.~D. Monteiro.
\newblock A proximal bundle variant with optimal iteration-complexity for a
  large range of prox stepsizes.
\newblock \emph{SIAM Journal on Optimization}, 31\penalty0 (4):\penalty0
  2955--2986, 2021.

\bibitem[Liang and Monteiro(2024)]{liang2024unified}
J.~Liang and R.~D. Monteiro.
\newblock A unified analysis of a class of proximal bundle methods for solving
  hybrid convex composite optimization problems.
\newblock \emph{Mathematics of Operations Research}, 49\penalty0 (2):\penalty0
  832--855, 2024.
\newblock \doi{10.1287/moor.2023.1372}.

\bibitem[Liang et~al.(2023)Liang, Monteiro, and Zhang]{liang2023proximal}
J.~Liang, R.~D.~C. Monteiro, and H.~Zhang.
\newblock Proximal bundle methods for hybrid weakly convex composite
  optimization problems, 2023.
\newblock URL \url{https://arxiv.org/abs/2303.14896v1}.

\bibitem[Liang et~al.(2024)Liang, Guigues, and Monteiro]{liang2024singlecut}
J.~Liang, V.~Guigues, and R.~D.~C. Monteiro.
\newblock A single cut proximal bundle method for stochastic convex composite
  optimization.
\newblock \emph{Mathematical Programming}, 208\penalty0 (1):\penalty0 173--208,
  2024.
\newblock \doi{10.1007/s10107-023-02035-2}.

\bibitem[Liao and Zheng(2026)]{liao2026revisiting}
F.-Y. Liao and Y.~Zheng.
\newblock Revisiting proximal bundle methods: Improved rates under {H{\"o}lder}
  smoothness, 2026.
\newblock URL \url{https://arxiv.org/abs/2609.15806v1}.

\bibitem[Liao et~al.(2026)Liao, Madden, and Zheng]{liao2026accelerated}
F.-Y. Liao, T.~Madden, and Y.~Zheng.
\newblock An accelerated proximal bundle method for convex optimization.
\newblock In \emph{Proceedings of The 8th Annual Learning for Dynamics and
  Control Conference}, volume 331 of \emph{Proceedings of Machine Learning
  Research}, pages 1012--1034. PMLR, 2026.
\newblock URL \url{https://proceedings.mlr.press/v331/liao26a.html}.

\bibitem[Lin and Zhang(2026{\natexlab{a}})]{partOne}
Z.~Lin and Z.~Zhang.
\newblock Accelerated prox-level methods for unknown piecewise-smooth
  optimization {I}: Convex optimization, 2026{\natexlab{a}}.
\newblock URL \url{https://arxiv.org/abs/2601.14680v3}.

\bibitem[Lin and Zhang(2026{\natexlab{b}})]{partTwo}
Z.~Lin and Z.~Zhang.
\newblock Accelerated prox-level methods for unknown piecewise-smooth
  optimization {II}: Function-constrained optimization, 2026{\natexlab{b}}.
\newblock URL \url{https://arxiv.org/abs/2609.03251v1}.

\bibitem[Lv et~al.(2018)Lv, Pang, and Meng]{lv2018constrained}
J.~Lv, L.~Pang, and F.~Meng.
\newblock A proximal bundle method for constrained nonsmooth nonconvex
  optimization with inexact information.
\newblock \emph{Journal of Global Optimization}, 70\penalty0 (3):\penalty0
  517--549, 2018.
\newblock \doi{10.1007/s10898-017-0565-2}.

\bibitem[Mifflin(1977)]{mifflin1977algorithm}
R.~Mifflin.
\newblock An algorithm for constrained optimization with semismooth functions.
\newblock \emph{Mathematics of Operations Research}, 2\penalty0 (2):\penalty0
  191--207, 1977.

\bibitem[Nesterov(2015)]{nesterov2015universal}
Y.~Nesterov.
\newblock Universal gradient methods for convex optimization problems.
\newblock \emph{Mathematical Programming}, 152\penalty0 (1):\penalty0 381--404,
  2015.

\bibitem[Nesterov(2018)]{nesterov2018lectures}
Y.~Nesterov.
\newblock \emph{Lectures on convex optimization}, volume 137.
\newblock Springer, 2018.

\bibitem[Ning and You(2020)]{ning2020transformation}
C.~Ning and F.~You.
\newblock A transformation-proximal bundle algorithm for multistage adaptive
  robust optimization and application to constrained robust optimal control.
\newblock \emph{Automatica}, 113:\penalty0 108802, 2020.
\newblock \doi{10.1016/j.automatica.2019.108802}.

\bibitem[Oliveira et~al.(2011)Oliveira, Sagastiz{\'a}bal, and
  Scheimberg]{deoliveira2011inexact}
W.~Oliveira, C.~Sagastiz{\'a}bal, and S.~Scheimberg.
\newblock Inexact bundle methods for two-stage stochastic programming.
\newblock \emph{SIAM Journal on Optimization}, 21\penalty0 (2):\penalty0
  517--544, 2011.
\newblock \doi{10.1137/100808289}.

\bibitem[Oliveira and Sagastiz{\'a}bal(2014)]{oliveira2014bundle}
W.~d. Oliveira and C.~Sagastiz{\'a}bal.
\newblock Bundle methods in the xxist century: A bird's-eye view.
\newblock \emph{Pesquisa Operacional}, 34\penalty0 (3):\penalty0 647--670,
  2014.

\bibitem[Pang et~al.(2023)Pang, Wang, and Meng]{pang2023nonconvex}
L.~Pang, X.~Wang, and F.~Meng.
\newblock A proximal bundle method for a class of nonconvex nonsmooth composite
  optimization problems.
\newblock \emph{Journal of Global Optimization}, 86:\penalty0 589--620, 2023.
\newblock \doi{10.1007/s10898-023-01279-8}.

\bibitem[Sagastiz{\'a}bal(2013)]{sagastizabal2013composite}
C.~Sagastiz{\'a}bal.
\newblock Composite proximal bundle method.
\newblock \emph{Mathematical Programming}, 140\penalty0 (1):\penalty0 189--233,
  2013.
\newblock \doi{10.1007/s10107-012-0600-5}.

\bibitem[W{\"a}chter and Biegler(2006)]{wachter2006implementation}
A.~W{\"a}chter and L.~T. Biegler.
\newblock On the implementation of an interior-point filter line-search
  algorithm for large-scale nonlinear programming.
\newblock \emph{Mathematical Programming}, 106\penalty0 (1):\penalty0 25--57,
  2006.

\bibitem[Wolfe(1975)]{wolfe1975method}
P.~Wolfe.
\newblock A method of conjugate subgradients for minimizing nondifferentiable
  functions.
\newblock In M.~L. Balinski and P.~Wolfe, editors, \emph{Nondifferentiable
  Optimization}, volume~3 of \emph{Mathematical Programming Studies}, pages
  145--173. North-Holland, Amsterdam, 1975.
\newblock \doi{10.1007/BFb0120703}.

\bibitem[Yang et~al.(2014)Yang, Pang, Ma, and Shen]{yang2014constrained}
Y.~Yang, L.~Pang, X.~Ma, and J.~Shen.
\newblock Constrained nonconvex nonsmooth optimization via proximal bundle
  method.
\newblock \emph{Journal of Optimization Theory and Applications}, 163\penalty0
  (3):\penalty0 900--925, 2014.
\newblock \doi{10.1007/s10957-014-0523-9}.

\bibitem[Zhang and Sra(2025)]{zhang2025linearly}
Z.~Zhang and S.~Sra.
\newblock Linearly convergent algorithms for nonsmooth problems with unknown
  smooth pieces, 2025.
\newblock URL \url{https://arxiv.org/abs/2507.19465v1}.

\end{thebibliography}

\end{document}